\documentclass[12pt,dvips,reqno]{amsart}
\usepackage{euler, amsfonts, amssymb, latexsym, epsfig,epic,stmaryrd}

\newtheorem{Theorem}{Theorem} 
\newtheorem{Proposition}{Proposition} 
\newtheorem{Lemma}{Lemma}

\newtheorem*{Corollary*}{Corollary}
 
\newtheorem*{Theorem*}{Theorem}
\theoremstyle{remark}
\newtheorem{Example}{Example}

\newcommand\lie[1]{{\mathfrak #1}}

\newcommand\Proj{{\rm Proj\,}}
\newcommand\iso{{\,\cong\,}}
\newcommand\tensor{{\otimes}}

\newcommand\onto{\mathop{\twoheadrightarrow}}

\newcommand\into{\operatorname*{\hookrightarrow}}
\newcommand\union{\cup}
\newcommand\Union{\bigcup}

\newcommand\Pone{{\mathbb P}^1}

\newcommand\complexes{{\mathbb C}}
\newcommand\integers{{\mathbb Z}}

\newcommand\naturals{{\mathbb N}}

\theoremstyle{plain}

\renewenvironment{quotation}
{\list{}{
    \setlength\itemindent{0em}%
    \setlength\leftmargin{1.5em}
    \setlength\rightmargin{1.5em}
  }%
\item[]}
{\endlist}

\newcommand\dfn{\bf} 

\newcommand\Spec{{\rm Spec}\,}
\newcommand\Gm{{\mathbb G}_m}
\newcommand\Ga{{\mathbb G}_a}

\newcommand\junk[1]{}

\newcommand\wtF{\widetilde F}

\begin{document}
\pagestyle{plain}

\title{Automatically reduced degenerations of \\
   automatically normal varieties}

\author{Allen Knutson}
\thanks{Supported by an NSF grant.}
\email{allenk@math.ucsd.edu}
\date{\today}

\maketitle

\begin{abstract}
  Let $F$ be a flat family of projective schemes, whose geometric generic
  fiber is reduced and irreducible. We give conditions on a
  special fiber (a ``limit'' of the family) 
  to guarantee that it too is reduced. 
  These conditions often imply also that the generic fiber is normal.
  The conditions are particularly easy to check in the setup of
  a ``geometric vertex decomposition'' [Knutson-Miller-Yong '07].

  The primary tool used is the corresponding limit {\em branchvariety} 
  [Alexeev-Knutson '06], which is reduced by construction, 
  and maps to the limit subscheme; 
  our technique is to use normality to show that the 
  branchvariety map must be an isomorphism.

  As a demonstration, we give an essentially na\"ive proof that 
  Schubert varieties in finite type are normal and Cohen-Macaulay. 
  The proof does not involve any resolution of singularities or
  cohomology-vanishing techniques (e.g. appeal to characteristic $p$).
\end{abstract}

{\small \tableofcontents}

\section{Statement of results}

\newcommand\PP{{\mathbb P}}
\renewcommand\AA{{\mathbb A}}
\newcommand\omicron{o}

Let $F \subseteq \PP^n \times S$ be a closed subscheme, flat over $S$,
considered as a family over $S$ of projective schemes. If $S$ is irreducible,
we can speak of the generic fiber of $F$, which throughout this paper
we assume to be (geometrically) reduced.

\subsection{General limits}
It is frequently useful to be able to guarantee that a particular fiber
$F_\omicron$ over a point $\omicron\in S$ is reduced.
Often its underlying set may be easy to calculate, but we may only
be able to check its reducedness generically, or in small codimension.

\newtheorem*{RLLemma}{Reduced Limit Lemma}

\begin{RLLemma}
  Let $F \subseteq \PP^n \times S$ be a flat family of $d$-dimensional 
  projective schemes (over a fixed Noetherian base scheme). 
  Let $S$ be irreducible and normal, and assume the generic fiber of
  $F\to S$ is {\em irreducible} (or at least,
  equidimensional and connected in codimension $1$)
  and geometrically reduced.

  Let $F_\omicron$ denote the fiber over a point $\omicron\in S$.
  Let $A_1,A_2,\ldots,A_k$ be the components of its reduction,
  automatically of dimension $d$. Assume that
  $F_\omicron$ is generically geometrically reduced and
  each $A_i$ is normal,
  and that at least one of the following holds:
  \begin{enumerate}
  \item $F_\omicron$ is irreducible ($k=1$),
  \item $F_\omicron$ has only two geometric components ($k=2$), 
    and $A_1 \cap A_2$ is reduced and irreducible,
  \item $F_\omicron$ is reduced through codimension $1$, and for each $i$,
    $A_i \cap (A_1 \union \cdots \union A_{i-1})$ is 
    equidimensional of dimension $d-1$, and reduced.
    (This may involve reordering the $\{A_i\}$.)
  \end{enumerate}
  Then $F_\omicron$ is reduced. 

  In case (1), the generic fiber is irreducible and normal.
  In case (2), if it is irreducible, it is normal.
  In cases (2) and (3) it at least satisfies Serre's condition $S_2$.
\end{RLLemma}

Case (1) is implied by a much older local version from \cite{Hironaka}; 
see also \cite{Kollar} and the references therein.
The conditions in the lemma seem difficult to weaken in cases (2) and (3),
as a few near-counterexamples may help demonstrate, though a
local version perhaps may be achievable using the results of \cite{rigid}.
In each of the following examples all conditions other than the
italicized one hold, but $F_\omicron$ is not reduced.
\begin{itemize}
\item {\em Generic fiber reducible.} Let $F_t$ be the union of
    two skew lines in $\PP^3$ at distance $t$ from one another;
    at $t=0$ let them cross at one point.
    In $F_0$, there is an embedded point at the crossing.
  \item {\em Special fiber not generically reduced.}
    Let a smooth plane conic degenerate to a double line. Then all other 
    conditions of case (1) hold.
  \item {\em $A_1\cap A_2$ reducible.} Let $X$ be a twisted cubic
    curve in $\PP^3$, degenerating to a planar union of a line and a
    conic, with an embedded point at one of the two points of
    intersection.  Then all other conditions of case (2) hold.
  \item {\em Nonprojective fibers.} 
    From the previous example,
    excise a generic $\PP^2$ passing through the other
    (the reduced) point of intersection.
    Then $X$ is a twisted cubic in $\AA^3$,
    degenerating to the union of a line and conic in $\AA^2$, 
    with an embedded point at the single point of intersection.
  \end{itemize}

  The third case in the Reduced Limit Lemma contains the first, as $k=1$.
  At $k=2$ it is slightly different from the second case; it is more
  generally applicable (in not requiring $A_1 \cap A_2$ irreducible)
  but harder to apply, in that one is required to check reducedness in
  codimension $1$ by other means.

  The condition ``$A_i \cap (A_1 \union \cdots \union A_{i-1})$ is
  equidimensional of dimension $d-1$'' also makes sense when the $\{A_i\}$ 
  are the facets of a simplicial complex; in that theory an ordering with 
  this property is called a {\em shelling}. 

We will recall what we need about Serre's conditions $S_k$ in section
\ref{ssec:serre}.
The conclusion of the Reduced Limit Lemma that the generic fiber
is $S_2$ has a simple extension:

\begin{Lemma}\label{lem:skgeneric}
  Assume the setup of the Reduced Limit Lemma (so in particular,
  $F_\omicron$ is reduced). Ask in addition that each component $A_i$
  of the special fiber $F_\omicron$ is $S_k$, and that each
  $A_i \cap (A_1 \union \cdots \union A_{i-1})$ is $S_{k-1}$.  
  Then $F_\omicron$ and the generic fiber are $S_k$.

  In particular, if each $A_i$ and 
  $A_i \cap (A_1 \union \cdots \union A_{i-1})$ are Cohen-Macaulay,
  then $F_\omicron$ and the generic fiber are Cohen-Macaulay.
\end{Lemma}

In the applications envisioned by the author (one of which will
occupy section \ref{sec:Schubert}), one starts with a general fiber,
constructs a one-parameter family over a punctured disc $S\setminus \omicron$,
and fills in the limit $F_\omicron$ by taking a certain closure. 
(Note that this construction requires that
the family be embedded, in order to have somewhere to take a closure.)
The Reduced Limit Lemma is then invoked to study this
{\em automatically flat} limit. A slightly different point of
view is taken in \cite{Kollar}, where one is given the family (so no
embedding is necessary) and one wants criteria to check whether it is flat.

\subsection{Geometric vertex decompositions}\label{ssec:GVDs}
We now describe a very specific sort of family which we proved some
results about already in \cite{KMY}. In this restricted case the
same techniques yield a stronger result, as the conditions to check
are particularly simple.

\newcommand\barX{\overline X}
Let $H\oplus L$ be a vector space, where $L$ is one-dimensional
(the letters are for Hyperplane and Line). 
Let $X \subseteq H\oplus L$ be a reduced, irreducible subvariety,
and consider its closure $\barX$ inside $H \times L\Pone$,
where $L\Pone = L\cup \{\infty\}$ denotes the projective completion of $L$.

Define the family
$$ F := \overline{ \{ (h,\ell,z) : (h,z^{-1} \ell) \in \barX \} }
\subseteq H\times L\Pone \times \AA^1,
\qquad\hbox{considered over }\AA^1. $$
If we let $\Gm$ act on $H\times L\Pone$ by scaling the second factor,
$z\cdot (\vec h,\ell) := (\vec h, z\ell)$, 
then $F_{z\neq 0} = z\cdot X$; 
every closed fiber but $F_0$ is isomorphic to $F_1 = \barX$.
This $F$ is automatically flat over $\AA^1$.

In the case that $F_0$ is reduced, we christened it a {\dfn geometric
  vertex decomposition} of $X$ in \cite{KMY}, as the splitting 
in equation (\ref{eq:gvdlemma}) below is
closely related to the splitting of a simplicial complex using a
``vertex decomposition''.

\newtheorem*{GVDlemma}{Geometric Vertex Decomposition Lemma}

\begin{GVDlemma}
  Let $\barX \subseteq H \times L\Pone$ be irreducible 
  and geometrically reduced, with $\Pi$ its projection to $H$ 
  and $\Lambda := \barX \cap (H\times \{\infty\})$. 
  Assume $\Lambda \subsetneq \barX$, i.e. $\barX$ is the closure
  of $X := \barX \cap (H\times L)$. Let
  $$ F := \overline{ \{ (h,\ell,z) : (h,z^{-1} \ell) \in \barX \} }
  \subseteq H\times L\Pone \times \AA^1. $$
  So far this is the general setup of \cite[theorem 2.2]{KMY},
  which asserts that
  \begin{equation}
    \label{eq:gvdlemma}
   F_0 = 
  (\Pi\times \{0\}) \cup_{\Lambda\times \{0\}} (\Lambda\times L\Pone)    
  \end{equation}
  as sets. Assume in addition that
  \begin{enumerate}
  \item the projection $\barX \to \Pi$ is generically $1$:$1$,
  \item $\Pi$ is normal, and
  \item $\Lambda$ is geometrically reduced.
  \end{enumerate}
  Then equation (\ref{eq:gvdlemma}) holds as schemes, i.e. $F_0$ is reduced,
  a geometric vertex decomposition of $X$.

  If $\Pi$ and $\Lambda$ are Cohen-Macaulay, then so are $F_0, \barX$, and $X$.
\end{GVDlemma}

In case (2) of the Reduced Limit Lemma, we required the intersection
$A_1 \cap A_2$ to be reduced; the analogue of this in the Geometric
Vertex Decomposition Lemma is the requirement that $\Lambda$ be reduced.
However, we do not need to require any analogues of the conditions 
that $A_2$ be irreducible and normal ($\Lambda$ may be neither),
and all the projectivity we need is in the $L\Pone$.

As in case (2) of the Reduced Limit Lemma, $\Lambda$ normal implies
that $\barX$ is normal. Indeed, this will be the case in our application
in section \ref{sec:Schubert}.
In other situations, though, $\Lambda$ is often only $S_2$,
so we give a more general criterion for normality of $\barX$:

\begin{Lemma}\label{lem:gvdnormal}
  Continue the situation of the Geometric Vertex Decomposition Lemma.

  Then every component of $\barX$'s singular locus 
  $\barX_{sing}$ of codimension $1$ in $\barX$
  is of the form $D\times L\Pone$,
  where $D \subseteq \Lambda_{sing}$ is codimension $1$ inside $\Lambda$.
  In particular, if $\dim \Lambda_{sing} < \dim \Lambda - 1$,
  there can be no such components $D$.

  If there are no such components, and $\Lambda$ is $S_2$, 
  then $\barX$ is normal. 
  In particular, $\Lambda$ normal implies $\barX$ also normal.
\end{Lemma}

To apply this lemma, one determines the components $D$ of $\Lambda_{sing}$
of codimension $1$ in $\Lambda$,
and checks that $\barX$ either does not contain or
is generically nonsingular along $D\times L\Pone$. 

After the proof of this lemma (in section \ref{sec:proofs}) we give an
example showing the criterion is necessary, in which $\barX_{sing}$
contains such a component and $\barX$ is not normal.

\subsection{Structure of the paper}
We prove all these lemmas in section \ref{sec:proofs}.
The crucial notion used is that of the {\em limit branchvariety} \cite{AK},
which is a sort of reduced avatar of the limit subscheme, 
but much better behaved than its simple reduction. 
(A similar ``correction'' already appears in \cite[remark 4.2]{Kollar}.)
In the cases at hand, though, the limit branchvariety and limit
subscheme coincide, showing the limit subscheme is reduced.
We recall these and other more standard 
notions in section \ref{sec:prelim}.

In section \ref{sec:Schubert} we apply these lemmas to give an
inductive proof of the well-known result (see e.g. \cite{R})
that Schubert varieties in arbitrary finite-dimensional flag manifolds are
normal and Cohen-Macaulay.
In very brief, we flatly degenerate an affine patch on a Schubert variety,
invoke the Geometric Vertex Decomposition Lemma to show the limit scheme is
a union of two simpler patches, and use induction. 
In particular, the proof does not involve any resolution of singularities or
cohomology-vanishing techniques (e.g. appeal to characteristic $p$),
and we expect it to apply to other families of subvarieties
of flag manifolds.

\junk{
  Is there a Billey-Willems formula for $T$-invariant multiplicity-free
  subvarieties? I guess it only really lets one reduce to the case
  of subvarieties that contain all the fixed points.
}

\subsection{Acknowledgements}
We thank Valery Alexeev, Michel Brion, David Eisenbud, Tom Graber,
Johan de Jong, Shrawan Kumar, Ezra Miller, and Ravi Vakil for useful
discussions.

\section{Geometric preliminaries}\label{sec:prelim}

In this section we assemble some standard geometric results, with the
more technical lemmas to wait until section \ref{sec:proofs}.

\subsection{Serre conditions}\label{ssec:serre}

A scheme $X$ is called {\dfn $S_m$ at the point $x\in X$} if the local
ring at $x$ possesses a regular sequence of length $m$,
perhaps after extension of the residue field. This is
equivalent to the vanishing of the local cohomology groups
$H^i_m(k[X])$ for $i<m$ \cite[chapter 6]{LC}.
If $X$ is $S_m$ at every point, we just say {\dfn $X$ is $S_m$}.

These properties are related to many familiar geometric ones,
particularly in tandem with the following conditions called $\{R_j\}$.
An equidimensional scheme $X$ is $R_j$ if its singular locus has 
codimension $>j$.  In particular, $X$ is generically reduced iff it is $R_0$.

\begin{Proposition}\label{prop:sk}
  Let $X$ be an equidimensional scheme.
  \begin{enumerate}
  \item $X$ is reduced iff $X$ is $R_0$ and $S_1$.
  \item $X$ is normal iff $X$ is $R_1$ and $S_2$ (Serre's criterion).
  \item $X$ is Cohen-Macaulay iff $X$ is $S_{\dim X}$. 
  \item If $X$ is the union $A \union B$ of two closed subschemes, 
    where $A,B$ are $S_k$ and $A\cap B$ is $S_{k-1}$, then $X$ is $S_k$.
  \end{enumerate}
\end{Proposition}

\begin{proof}
  \begin{enumerate}
  \item Exercise 11.10 of \cite{Ei}.
  \item Theorem 11.5 of \cite{Ei}.
  \item This is the usual definition.
  \item 
    For this we use the Mayer-Vietoris sequence on local cohomology
    \begin{align*}
      \ldots&\to& H^{k-1}_m(k[A \cap B]) &\to& H^{k-1}_m(k[A]+k[B]) &\to& H^{k-1}_m(k[A \cup B]) && \\
            &\to& H^{k  }_m(k[A \cap B]) &\to& H^{k  }_m(k[A]+k[B]) &\to& H^{k  }_m(k[A \cup B]) &\to&\cdots
    \end{align*}
    from \cite[chapter 3]{LC} to infer the necessary vanishing.
  \end{enumerate}
\end{proof}

\subsection{Flat families of projective varieties}

We record a proposition, seemingly well-known to the experts,
concerning the two-way flow of information between special and generic
fibers in a flat family.

\begin{Proposition}\label{prop:degens}
  Let $F \subseteq \PP^n \times S$ be a flat family over $S$ of 
  projective schemes.
  Assume that the base $S$ is irreducible and normal.
  \begin{enumerate}
  \item If the generic fiber is nonempty, then each special 
    fiber is also nonempty.
  \item If a special fiber satisfies Serre's condition $S_k$, 
    then the generic fiber does too.
  \item If the reduction of the generic fiber of $F$ is equidimensional,
    then the reduction of any fiber is, and the dimensions match.
  \item If the reduction of the generic fiber of $F$ is connected in
    codimension $1$, then the reduction of any fiber is.
  \end{enumerate}
\end{Proposition}

\begin{proof}
  The map $F\to S$ hits a nonempty open set since the generic
  fiber is nonempty. Being also proper, the map is onto, 
  so each special fiber is nonempty.
  
  For the second claim, let $C_\eta \subseteq F_\eta$ be the
  non-$S_k$-locus of the generic fiber, and $C \subseteq F$ its
  closure to a flat family. 
  Since $S_k$ is a cohomology-vanishing condition and cohomology
  groups are semicontinous, $C_\omicron \subseteq F_\omicron$.
  By part (1), $C_\omicron$ empty implies $C_\eta$ empty.

  We now combine the technique of slicing with general planes 
  (perhaps after harmlessly extending the base field coming from
  the point $\omicron \in S$), and the following version
  of Zariski's Main Theorem: if a flat family of complete schemes 
  over a normal base has connected generic fiber, then all
  fibers are connected (see e.g. \cite{Chow}). 

  If the reduction of some fiber has a component of small dimension,
  we slice with a general plane to replace that component by points.
  Now the generic fiber is still irreducible
  (Bertini's theorem) hence connected, but the special fiber is disconnected,
  contradiction. This proves the third claim.
  
  To prove the fourth claim, slice with a general plane to replace $F$ 
  with a family of curves. The general fiber of this subfamily is still
  connected, so the special fiber of this subfamily is connected,
  hence the original special fiber was connected in codimension $1$.
\end{proof}

There are other contexts where part (1) holds, e.g.
the case that $F$ is a 
family of $\Gm$-invariant subschemes of $\PP^n \times \AA^k$, where $\Gm$ 
acts linearly on $\PP^n$ and $\AA^k$ and has only positive weights on $\AA^k$. 
With some work, one can extend this proposition (and the Reduced
Limit Lemma, which depends on it) to that context.

\subsection{Branchvarieties}

We recall the basic construction from \cite{AK} 
of a {\em limit branchvariety}. 

A {\dfn branchvariety $X$ of $Y$} is a map $\beta: X \to Y$ of schemes
such that $\beta$ is finite (proper with finite fibers) and $X$ is
(geometrically) reduced. In particular, any closed reduced subscheme 
of $Y$ is a branchvariety of $Y$; the prefix {\em branch} should be seen as
analogous to {\em sub}. 
The basic facts we need about branchvarieties
are collected in the following:

\begin{Theorem}
  \label{thm:AK}
  Let $F \to S$ be a flat family of subschemes of $Y$, 
  where $S$ is a normal one-dimensional base,
  for example $\Spec$ of a discrete valuation ring.
  Assume the fiber $F_\omicron$ is generically geometrically reduced, 
  and that all other fibers are geometrically reduced. 
  Let $S^\times := S\setminus \omicron$,
  and let $F^\times$ denote the restriction to $S^\times$.

  Then there exists uniquely a flat family $\wtF \to S$ 
  of branchvarieties of $Y$ extending $F^\times \to S^\times$, 
  and a natural finite map $\beta_\omicron: \wtF_\omicron \to F_\omicron$ 
  whose image is the reduction $(F_\omicron)_{red}$.
  This $\wtF$ may be constructed as the normalization of $F$ in
  the open set $F^\times$.
  
  This $\beta_\omicron$ induces a correspondence between the
  top-dimensional components of $\wtF_\omicron$ and $F_\omicron$, and
  is generically $1$:$1$ on each top-dimensional component.
  \junk{
    The degree of $\beta$, over a component $C$ of $F_\omicron$, is the 
    length of $F_0$ at the generic point of $C$. In particular $\beta$ is
    generically $1$:$1$ over $C$ iff $F_\omicron$ is generically reduced
    along $C$.
  }
\end{Theorem}

\begin{proof}
  The first two paragraphs are theorem 2.5 and corollary 2.6 of \cite{AK};
  the base change usually required in \cite[theorem 2.5]{AK} may be omitted by
  the assumption that $F_\omicron$ is generically geometrically reduced 
  (so each $m_i=1$ in the notation of \cite[theorem 2.5]{AK}).
  
  If $C^\times$ is a component of $F^\times$ not of top dimension, then
  its closures $C,\widetilde C$ inside $F,\wtF$ are flat subfamilies, 
  whose special fibers $C_\omicron,\widetilde C_\omicron$ are 
  therefore also not of top dimension (by proposition \ref{prop:degens}).
  So we can safely remove these components of $F$ without affecting
  the top-dimensional components of $F_\omicron,\wtF_\omicron$.
  Hereafter we work with the unions $F',\wtF'$ 
  of the top-dimensional components of $F$; call this dimension $d$.

  Then, again by proposition \ref{prop:degens},
  we find that $\wtF'_\omicron,(F'_\omicron)_{red}$ are also
  equidimensional of dimension $d$.
  By the finiteness of $\beta$, the image of a $d$-dimensional component
  of $\wtF'_\omicron$ is again of dimension $d$, hence a component
  of $F'_\omicron$. So far we have a function 
  from the 
  set of components of $\wtF'_\omicron$ (i.e. the top-dimensional 
  components of $\wtF_\omicron$) to the
  set of components of $F'_\omicron$ (i.e. the top-dimensional 
  components of $F_\omicron$).

  The map $\beta$ induces a top-degree Chow class 
  $\beta_*([\wtF_\omicron])$ on $F_\omicron$, which we can compute as 
  $$ \beta_*([\wtF_\omicron]) 
  = \sum_{D \subseteq \wtF_\omicron} \beta_*([D])
  = \sum_{D \subseteq \wtF_\omicron} 
  [\beta(D)] \deg\big(D \to \beta(D)\big) 
  = \sum_{E \subseteq F_\omicron} 
  [E] \sum_{D \subseteq \wtF_\omicron, \beta(D) = E}
  \deg\big(D \to E\big) $$
  where the sums are over top-dimensional reduced components,
  and $[Z]$ denotes the fundamental Chow class of the scheme $Z$.
  However, the Chow class shadow of the
  much more precise $K$-class statement \cite[proposition 6.1]{AK}
  tells us that $\beta_*([\wtF_\omicron]) = [F_\omicron]$,
  a fact already used in \cite{K06} in the case that $F$ is a
  degeneration to a normal cone.

  Finally, the fact that $F_\omicron$ is generically reduced tells us that
  its fundamental Chow class is simply
  $$ [F_\omicron] = \sum_{E \subseteq F_\omicron} 1\cdot [E], $$
  so for each top-dimensional component $E$ of $F_\omicron$, we have
  $$ \sum_{D \subseteq \wtF_\omicron, \beta(D) = E}
  \deg\big(D \to E\big) = 1. $$
  Hence there is only one component $D$ mapping to $E$, and the degree
  of the map is $1$.
\end{proof}

In fact the construction in theorem \ref{thm:AK} does not require the
assumption of generic reducedness of $F_\omicron$; the only
modification necessary is a certain ramified base change $S' \onto S$.
Since we assume generic reducedness in the Reduced Limit Lemma,
we didn't state here that only slightly more complicated 
but much more general result (which can be found in \cite{AK}).
We mention, though, that in that more general setup the map induced on the
sets of top-dimensional components still exists
but may be only surjective (as in example 3 of \cite{K06}).

By the uniqueness of the limit branchvariety, if the
limit subscheme is reduced, then it agrees with the limit branchvariety.
We now sharpen this to a local statement (that again, does not actually
require generic reducedness).

\begin{Lemma}\label{lem:crucial}
  Assume the setup of theorem \ref{thm:AK},
  and let $U_\omicron \subseteq F_\omicron$ be an open subset.
  Then the map 
  $\beta: \beta^{-1}(U_\omicron) \to U_\omicron$
  is an isomorphism iff $U_\omicron$ is reduced.
\end{Lemma}

\begin{proof}
  We may pick an open set $U\subseteq F$ 
  such that $U \cap F_\omicron = U_\omicron$.
  Then the lemma can be rephrased as ``for every $U\subseteq F$...''
  Now observe that the lemma holds for $U$ iff it holds for an open cover, 
  so it is enough to handle the case $U$ affine.

  Recall now the construction of $\wtF$: it is the
  normalization of $F$ in the open set $F \setminus F_\omicron$.
  Since normalization commutes with localization to open sets, 
  we see that $\beta^{-1}(U)$ is the normalization of $U$ in the open set
  $U \setminus U_\omicron$. 

  Part (1) of \cite[lemma 2.1]{AK} now says that $U_\omicron$ reduced
  implies that $\beta: \beta^{-1}(U) \to U$ is an isomorphism 
  (and in particular, induces an isomorphism of the fibers over $\omicron$).

  Part (2) of \cite[lemma 2.1]{AK} only says that $U_\omicron$ nonreduced
  implies that after some base change, which can change the normalization,
  does $\beta: \beta^{-1}(U') \to U'$ fail to be an isomorphism.
  In the case at hand, since $\beta^{-1}(U_\omicron)$ is reduced, 
  as in \cite[corollary 2.6]{AK} the base change can only extend the
  residue field, so the map is already not an isomorphism.
\end{proof}

This lemma gives a way to show $F_\omicron$ is reduced without
studying $F_\omicron$ directly; instead we may show that 
$\beta_\omicron: \wtF_\omicron \onto (F_\omicron)_{red}$ is an isomorphism.

\newcommand\FF{\mathbb F}
One of the very few surprises in moving beyond subvarieties of projective space
to branchvarieties is the failure of some Bertini theorems in characteristic 
$p$: for example the Frobenius map $\PP^1_{\FF_p} \to \PP^1_{\FF_p}$
is a branchvariety whose every hyperplane section is nonreduced.
In \cite[assumption 7.2]{AK} we got around this by assuming the
characteristic was $0$ or large enough, but we take a different tack here:

\begin{Lemma}\label{lem:bertini}
  Let $\beta: X \to \PP^n$ be a branchvariety, defined over a field,
  that is birational on each component.
  Then for a general plane $P \subseteq \PP^n$ (which may require
  extending the field), the ``plane section'' $\beta^{-1}(P) \subseteq X$ 
  is reduced, and itself a branchvariety of $\PP^n$ that is
  generically $1$:$1$ on each component.
\end{Lemma}

\begin{proof}
  We may assume that $P$ is a hyperplane, as we can
  then use induction. 

  Since the map $X \to \beta(X)$ is birational on each component,
  it is unramified. So by \cite[Thm. 6.3 (3)]{Jouanolou}, a generic
  plane section $\beta^{-1}(P)$ of it is again geometrically reduced.

  \junk{
    What the usual Bertini theorems {\em do} guarantee us is that
    a general hyperplane section $\beta^{-1}(P)$ is $S_1$
    \cite[proposition 4.2]{EGA5}, and that $P \cap \beta(X)$ is reduced.
    By assumption, the map from $X \to \beta(X)$ is birational
    on each component, so the map $\beta^{-1}(P) \to P \cap \beta(X)$ is also.
    Hence $\beta^{-1}(P)$ is $S_1$ and generically reduced, therefore reduced.
    Now we use the algebraic closure of the base field to ensure it is
    geometrically reduced, as required in the definition of branchvariety.
  }
    
  A proper map $C \to D$ of irreducible varieties is generically $1$:$1$
  if {\em some} fiber is a (reduced) point; then the set of $d\in D$
  for which the fiber is a point is open in $D$.
  Since $P$ intersects this open locus in each component of $\beta(X)$,
  we see that $\beta^{-1}(P) \to P\cap \beta(X)$ is again generically $1$:$1$
  on each component.
\end{proof}

\subsection{Geometric vertex decompositions}

We described the setup, $F_1 = \barX \subseteq H \times L\Pone$ degenerating to
$F_0$, in section \ref{ssec:GVDs}. We now collect (and slightly refine)
the results we will need from \cite{KMY}, which partially describe $F_0$.

\begin{Theorem}\label{thm:GVDs}
  Let $X$ be a closed subscheme of $H\times L$, where $H$ is a hyperplane
  and $L$ is a line, and let $\barX$ be its closure in $H\times L\Pone$. 
  Let $\Pi \subseteq H$ be the image of $\barX$ under projection to $H$,
  and define $\Lambda \subseteq H$ 
  by $\Lambda \times \{\infty\} := \barX \cap (H\times \{\infty\})$.
  Consider the family
  $$ F := \overline{ \{ (h,\ell,z) : (h,z^{-1} \ell) \in \barX \} }
  \subseteq H\times L\Pone \times \AA^1 $$
  automatically flat over the $\AA^1$ factor.
  Then as sets,
  $$   F_0 = 
  (\Pi\times \{0\}) \cup_{\Lambda\times \{0\}} (\Lambda\times L\Pone) $$
  and the two agree as schemes away from $\Pi \times \{0\}$.

  If $X$ is irreducible and the projection $X\to \Pi$ is 
  generically $1$:$1$, then $F_0$ is generically reduced along $\Pi$.

  All of this holds if $H$ is not a vector space, 
  but is merely quasiprojective.
\end{Theorem}

\newcommand\<{\langle}
\renewcommand\>{\rangle}

\begin{proof}
  This will be a slight variation of \cite[theorem 2.2]{KMY}, in turn
  based on the algebra from \cite[theorem 2.1]{KMY}, which uses 
  coordinates $\{x_1,\ldots,x_n\},\{y\}$ on $H,L$.
  Let $I$ be the ideal defining $X$.
  That theorem makes use of a 
  Gr\"obner basis $\{y^{d_i} q_i + r_i \mid i=1\ldots m\}$ of $I$, 
  with respect to a term order that picks out a term from the 
  initial $y$-form $y^{d_i} q_i$ of $y^{d_i} q_i + r_i$. 
  Theorem 2.1 also defines the ideals
  $$  I' = \<y^{d_i} q_i \mid i = 1,\ldots,m\>, \quad C = \<q_i
      \mid i = 1,\ldots,m\>, \quad P = \<q_i \mid d_i = 0\> + \<y\> $$

  For our first step, we introduce a coordinate $y'$, the denominator 
  coordinate 
  on $L\Pone$. Then we homogenize the generators of $I$ in $\{y,y'\}$, 
  meaning that each term in each $r_i$ is multiplied
  by the right power of the new $y'$ to make the generator
  $y^{d_i} q_i + r_i$ homogeneous in $\{y,y'\}$. This is the algebraic
  counterpart of defining $\barX$ as the closure of $X$.
  Call this $\{y,y'\}$-homogeneous ideal $I_h$.

  Then $\infty \in L\Pone$ is defined by the equation $y'=0$,
  so $\Lambda \times \{\infty\} := \barX \cap (H \times \{\infty\})$ 
  is defined by the ideal $I_h + \<y'\>$.
  Projecting to $H$ amounts to inverting $y$ and dropping the variables $y,y'$,
  which gives us the ideal $C$. If we reintroduce $y,y'$ as free variables,
  we get the ideal defining $\Lambda \times L\Pone$.
  
  As was observed in \cite[theorem 2.2]{KMY}, the limit $F_0$ is defined by 
  the ideal $I'$. Upon inverting $y$, the ideals $I'$ and $C$ coincide,
  which is the statement that $F_0$ and $\Lambda\times L\Pone$ agree
  (as schemes) away from $H \times \{0\}$.

  We can study $F_0$ away from $H\times \{\infty\}$ by passing to $y'=1$;
  this recovers the affine situation in \cite[theorem 2.2]{KMY},
  which tells that 
  $$ F_0 \setminus (H\times \{\infty\}) =
  (\Pi\times \{0\}) \cup_{\Lambda\times \{0\}} (\Lambda\times L) 
  \qquad \text{(not $L\Pone$)} $$
  as sets. 

  If $X$ is irreducible and the map $\barX \to \Pi$ is generically $1$:$1$,
  then it is a degree $1$ map,
  and from \cite[theorem 2.5]{KMY} we learn that 
  $F_0 \setminus (H\times \{\infty\})$ 
  is generically reduced along $\Pi\times \{0\}$.
  Then the same statement holds for $F_0$.

  {\bf $H$ quasiprojective rather than linear.}
  The stated result makes sense for $H$ an arbitrary scheme,
  not just a vector space, and it is easy to see that 
  \begin{enumerate}
  \item if the result holds for a scheme $H$, and $H'$ is
    a subscheme with $H \supseteq H' \supseteq \Pi$, then the result
    holds for $H'$
  \item if the result holds for each patch in an open cover of $H$,
    then it holds for $H$.
  \end{enumerate}
  By (1) one can reduce to the case that $H$ is projective space,
  and by (2) one can reduce to the already treated case 
  that $H$ is affine space.
\end{proof}

The quasiprojectivity assumption seems very unlikely to be necessary.
We did not pursue its removal for two reasons:
to do so would involve extending the theory of Gr\"obner bases 
beyond polynomial rings (or replacing the argument altogether), and 
our application in section \ref{sec:Schubert} only uses $H$ linear anyway.

\section{Proofs}\label{sec:proofs}

\subsection{Preliminaries}\label{ssec:prelim}

We start with a lemma about gluing schemes together along closed subschemes.
\newcommand\tX{\widetilde X}

\begin{Lemma}\label{lem:gluing}
  Let $A,B,X$ be schemes with a map $A\coprod B \onto X$
  such that $A \to X$, $B \to X$ are embeddings; hence we can identify
  $A,B$ with their images in $X$. Let $C$ be the intersection of
  $A$ and $B$ in $X$. Then $X$ (plus the inclusions $A,B \to X$)
  is determined up to unique isomorphism by $C \subseteq A,B$.

  Moreover, if the map factors as $A \coprod B \onto X' \onto X$, 
  with $C' = A\cap_{X'} B$, then $C' \subseteq C$, with equality iff
  the map $X' \to X$ is an isomorphism.
\end{Lemma}

\begin{proof}
  If $U \subseteq X$ is open, then $(A\cap U)\coprod (B\cap U) \to U$
  satisfies the same conditions. Conversely, if the statement holds for
  each $U$ in an open cover of $X$, then it holds for $X$. So we can
  restrict to the case $A,B,X$ affine, with $X = \Spec R$.

  Let $I_A,I_B,I_C$ be the ideals defining $A,B,C$.
  Then $I_C = I_A+I_B$ by definition. 
  The condition $X = A\cup B$ says that $I_A \cap I_B = 0$.
  Then $R$ is the inverse limit of $R/I_A, R/I_B \to R/I_C$, 
  and hence determined up to unique isomorphism by $C \subseteq A,B$.
  
  For the second claim, consider the diagram $A,B \to X' \to X$.
  Then the pullback $C'$ of $A,B\to X'$ automatically maps to 
  the pullback $C$ of $A,B\to X$, and since the inclusion $C' \to A$ 
  factors as $C' \to C \to A$, the map $C' \to C$ is an inclusion.
  By the first claim, $X,X'$ determine and are determined by the
  subschemes $C,C'$ of $A$ and $B$, so the map $X' \to X$ is an
  isomorphism iff the inclusion $C' \into C$ is an isomorphism.   
\end{proof}

The next lemma will be our source of normality for a generic fiber.
We take a moment to recall the difference between the {\em generic} fiber
of a family over an irreducible base $S$, which is the fiber over the
generic point of $S$, and a {\em general} fiber, whose definition
only makes sense if $S$ has enough closed points to have ``general'' ones.
In particular, $S$ should not be local, and should typically be defined
over an infinite field.

A general fiber of $\AA^1 \to \AA^1, z\mapsto z^2$ is reducible, but the 
generic fiber is the generic point of the source $F = \AA^1$, so irreducible.
A tighter analogue is provided by the {\dfn geometric generic fiber}
of $F\to S$, made by base-changing $F$ using the algebraic closure
of the function field of $S$. In particular, while the general fibers
and the geometric generic fiber behave well under many base changes,
the generic fiber can go from irreducible to reducible.

For a reduced complete (though possibly disconnected) 
curve $C$ with at worst nodal singularities,
let $\Gamma(C)$ denote its {\dfn graph of components} (as in e.g.
\cite{Oda}), with vertex set the set of components of $C$, and edge
set the set of nodes of $C$. There may be multiple edges between 
two vertices, and a singular component gives a vertex with self-edges.
The graph is connected iff the curve itself is.

\begin{Lemma}\label{lem:nodes}
  Let $F \to S$ be a flat family of at-worst-nodal geometrically reduced 
  curves over an irreducible normal base. 
  Assume that $F_{sing} \to S$ is proper, e.g. if $F$ itself is.
  Then there is a natural injection 
  $ \Sigma : edges(\Gamma(F_\eta)) \to edges(\Gamma(F_\omicron)) $
  from the nodes of the generic fiber to the nodes of the special fiber.

  Now assume that $F\to S$ itself is proper.
  If the generic fiber is connected, then $\Gamma(F_\omicron)$ is connected.
  If the geometric generic fiber is irreducible,
  then $\Gamma(F_\omicron)$ remains connected even when the edges in
  the image of $\Sigma$ are removed.
\end{Lemma}

\begin{proof}
  The finite map $F_{sing} \to S$ may be ramified; 
  perform base changes around the special fiber to make it unramified. 
  By the assumptions (used here only) that the geometric generic fiber
  is irreducible and the special fiber is geometrically reduced, after
  this finite base change the generic fiber will stay irreducible (if it was)
  and the special fiber will stay reduced.

  Given a node $N_\eta$ in the generic fiber, take its closure in $F$
  to get a subfamily $N$ lying in the singular locus $F_{sing}$. 
  Define $\Sigma(N_\eta) := N_\omicron \in (F_\omicron)_{sing}$,
  which exists by the assumed properness of $F_{sing} \to S$.
  This map $\Sigma$ is an injection, since 
  $N \cap N' \neq \emptyset$ implies $(N \cup N')_\eta$ is a fat point
  sitting inside $(F_\omicron)_{sing}$, but that is reduced.
  
  If the generic fiber is connected, then by proposition \ref{prop:degens}
  the special fiber is too, making its graph connected.
  In the rest we assume that the generic fiber (after the above
  base change) is irreducible, which is implied by the geometric
  generic fiber being irreducible.

  By assumption $N \to S$ is unramified.
  A formal neighborhood of $N$ inside $F'$ is essentially a deformation
  over $S$ of the singularity $\{xy=0\}$. It is easy to compute the
  universal deformation $\{xy=t\}$ of this formal singularity, and show that
  the formal neighborhood of $N$ is a trivial family over $S$
  (for any $t\neq 0$, the deformation would smoothe entirely).
  So if we blow up $F$ along $N$, it simply detaches that node
  in each fiber.\footnote{%
    In general, blowing up a flat family along a flat subfamily does
    not commute with passage to fibers. For example, if we blow up the
    family $\{(x,y,z) : xy = z^2\}, (x,y,z) \mapsto z$ along the
    section $\{(z,z,z)\}$, the $z=0$ fiber $\{xy=0\}$ acquires a new $\Pone$
    component connecting the now-disjoint axes $x=0,y=0$.}

  We now blow up $F$ along every subfamily $N$ coming from a node of $F_\eta$.
  (Note that no two 
  intersect, or else the singular locus of $F_\omicron$ would be
  nonreduced where two collided, but it is reduced. So there is no
  worry about the order in which they are blown up.) 
  The generic fiber stays irreducible under these blowings-up, 
  hence connected, so by proposition \ref{prop:degens} (which requires
  the normal base) the special fiber and its graph stay connected.
  Its new graph is the old one $\Gamma(F_\omicron)$ with the
  edges in the image of $\Sigma$ removed.
\end{proof}

One easy corollary of this is that if $F$ is a family of projective
curves and $\Gamma(F_\omicron)$ is a tree,
and the geometric generic fiber is irreducible, it can have no nodes
and must be normal.
There is another proof of this, explained to us by Johan de Jong
and Valery Alexeev. 
The Jacobian of an at-worst-nodal curve is an abelian variety iff its
graph is a tree (see \cite[Proposition 10.2]{Oda}).
The locus of abelian varieties is open in the Picard scheme, so the
condition of having an abelian variety as one's Jacobian is an
open condition on the fibers. Hence $\Gamma(F_\eta)$
is also a tree. By the assumption of irreducibility, that graph 
has only one vertex, and by treeness, no self-edges. So the generic
fiber is one normal component, QED.

\junk{
  Unfortunately, when we slice a higher-dimensional example, the
  associated graph has many edges between two vertices and is not a tree,
  so we need the following more involved argument.
}

We now prove a higher-dimensional version of that corollary,
for which we didn't see a Jacobian-based proof.

\begin{Lemma}\label{lem:genfibnormal}
  Let $F \to S$ be a flat family of reduced projective schemes over 
  a normal irreducible base. Assume that the geometric generic fiber 
  is irreducible.

  Assume that a special fiber $F_\omicron$ has only normal components
  $C_1,\ldots,C_n$, where for each $i$ there exists $J(i)<i$ such that
  $$ C_i \cap (C_1 \union \ldots \union C_{i-1}) = C_i \cap C_{J(i)} $$
  and this intersection $C_i \cap C_{J(i)}$ is irreducible and
  generically reduced.  Finally, assume that $C_i \cap C_j \cap C_k$
  has dimension at most $\dim F_\omicron - 2$ for $i,j,k$ distinct.

  Then the generic fiber is normal.
\end{Lemma}

\begin{proof}
  \junk{
    For purposes of discussion more than anything else, we replace $S$
    by the spectrum of a discrete valuation ring, the germ of a smooth
    curve through the point under the special fiber considered.
    This doesn't affect the special, generic, or geometric generic fibers.
  }
  Parallelling the case of curves, we define a graph $\Gamma$ whose
  vertices are $\{1,\ldots,n\}$ and with an edge between $i$ and $j$
  iff $\dim (C_i \cap C_j) = \dim F_\omicron - 1$. Then since the
  geometric generic fiber is irreducible, by
  proposition \ref{prop:degens} the special fiber is connected in
  codimension $1$, making this graph $\Gamma$ connected.
  The conditions on the intersections, that when listing the vertices
  in order each attaches to a unique previous one,
  imply that $\Gamma$ is a tree.
  (The converse is not quite true -- 
  imagine $\PP^2 \cup_{\Pone} \PP^2 \cup_{\Pone} \PP^2$ where the first
  and third component meet in two points.)

  To use Serre's criterion, we must show the generic fiber $F_\eta$
  is $S_2$ and $R_1$.
  We prove $C_1 \cup \ldots \cup C_i$ is $S_2$ by induction on $i$
  (the $i=1$ case being trivial):
  $$ C_1 \cup \ldots \cup C_i 
  = (C_1 \cup \ldots \cup C_{i-1}) \cup_{C_i \cap C_{J(i)}} C_i $$
  where the right-hand-side is the union of two $S_2$ schemes along 
  a reduced scheme. Then proposition \ref{prop:sk} says that this scheme
  is $S_2$. For $i=n$, we learn that the special fiber is $S_2$,
  so by proposition \ref{prop:degens} the geometric generic fiber is $S_2$.

  {\bf Slicing down to a family of curves.}
  Pick a plane $P$ in general position with respect to the special fiber
  and the generic fiber (by extending the base field if necessary),
  of complementary dimension to $F_\omicron$ plus one.
  To show that the generic fiber is $R_1$, we want to show that its
  intersection with $P$ is $R_1$, i.e. that it is a normal curve.
  The precise generality conditions we want are that
  \begin{itemize}
  \item each $P \cap C_i$ is a normal curve,
  \item each $P \cap C_i \cap C_{J(i)}$ is a set of reduced points, 
  \item each $P \cap C_i \cap C_j \cap C_k$, 
    for $i,j,k$ distinct, is empty, and
  \item $P$'s intersection with the geometric generic fiber is a 
    reduced and irreducible curve.
  \end{itemize}
  Let $F'$ be the family of curves given by intersecting every
  fiber with $P$, and $C'_i = P \cap C_i$ the components of $F'_\omicron$.
  It is worth noting that $F'$ does {\em not} satisfy one of the conditions
  we required of $F$, namely that $C_i \cap C_{J(i)}$ is irreducible.
  Rather, the corresponding intersection $C'_i \cap C'_{J(i)}$ is a finite
  set of points.

  The special fiber $F'_\omicron$ is a union of normal curves $\{ C'_i \}$.
  Each $C'_i \cap C'_j = P \cap C_i \cap C_j$, $j<i$,
  is only nonempty if $j = J(i)$, in which case it is reduced:
  $F'_\omicron$ has only ordinary double points.
  The graph $\Gamma(F'_\omicron)$ is almost the same as
  the graph $\Gamma$ constructed above; the only difference is that
  two connected vertices will, as explained in the last paragraph,
  usually have {\em many} edges between them. 
  
  \junk{
  Since $F'_\omicron$ is a curve with at worst nodal singularities, 
  the same is true of $F'_\eta$. We now claim that the map
  $F'_{sing} \to S$ is finite, i.e. it doesn't contain 
  any components of $F'_\omicron$ or $F'_\eta$, because they
  are generically reduced. 
    If it did, then the surface
    $F'$ would be singular in codimension $1$ along $F'_\omicron$, 
    and we could resolve these singularities by normalizing $F'$
    in the open set $F'_\eta$. This is exactly the construction
    in theorem \ref{thm:AK} used to replace a nonreduced fiber
    by a branchvariety; but by lemma \ref{lem:crucial} and the
    assumption that $F'_\omicron$ is generically reduced,
    this normalization doesn't affect $F'_\omicron$ generically.
    Unwinding, we see that $F'$ is generically regular along $F'_\omicron$,
    so doesn't contain any of its components.
  }

  {\bf Using nodes in $F'_\eta$ to disconnect $F'_\omicron$.}
  If $F'_\eta$ is regular, we are done. 
  Otherwise we may pick a singular point (a node) in $F'_\eta$, 
  and take its closure $N$ in $F'$. 
  Then $N_\omicron$ is a node in $F'_0$, so a point in
  $C'_i \cap C'_{J(i)}$ for some unique $i$. This $N$ and $i$
  are fixed hereafter.

  We now claim that for {\em every} point (node) $p \in C'_i \cap C'_{J(i)}$, 
  there exists a section $N^p \subseteq F'_{sing}$ of $F'\to S$ with
  $N^p_\omicron$ the desired node. This follows from the assumed
  irreducibility of $C_i \cap C_{J(i)}$, as follows.
  By varying the general plane $P$
  we can vary the intersection $C'_i \cap C'_{J(i)}$, and the nodes
  for which there do, resp. don't, exist such sections sweeps out 
  an open set in $C_i \cap C_{J(i)}$. By the irreducibility,
  one of these two open sets is empty; since we used $N$ to choose $i$
  we know it is the set of nodes for which there don't exist such sections.

  This says that the map $\Sigma$ from lemma \ref{lem:nodes} surjects
  onto the edges connecting $C'_i,C'_{J(i)}$ in the graph of $F'_\omicron$.
  Removing those edges disconnects the graph, counter to the
  result of lemma \ref{lem:nodes}. This contradiction traces back to
  our assuming that $F'_\eta$ was not regular.
\end{proof}

Our last technical lemma contains a couple of simple observations about the 
families in theorem \ref{thm:GVDs} concerning geometric vertex decompositions.

\begin{Lemma}\label{lem:gvdslice}
  For any $Y \subseteq H \times L\Pone$ as in theorem \ref{thm:GVDs},
  let $F(Y)$ denote the flat family constructed there.
  Let $\barX \subseteq H \times L\Pone$ be the closure 
  of $X \subseteq H\times L$. Then
  \begin{enumerate}
  \item $X$ is nonempty iff $F(\barX)_\omicron$ is nonempty.
  \item If $\barX$ is an irreducible curve satisfying the conditions
    of theorem \ref{thm:GVDs}, and its projection $\Pi \subseteq H$ is normal,
    then $\barX$ is itself normal.
  \end{enumerate}
\end{Lemma}

\begin{proof}
  \begin{enumerate}
  \item If $X$ is nonempty, then $\barX$ and its projection $\Pi$ are nonempty,
    and $F(\barX)_\omicron \supseteq \Pi \times \{0\}$ so it too is
    nonempty. The converse is obvious.
    \junk{
    \item The fiber over $1$ of the two families is the same, so the
      families agree away from the zero fiber. Both are closed in
      $P \times L\Pone \times \AA^1$, but the first is defined as a closure, so
      $$ F\left(\barX \cap (P\times L\Pone)\right) 
      \subseteq F(\barX) \cap (P \times L\Pone \times S).$$
    }
  \item
    The argument from lemma \ref{lem:nodes} must be modified slightly,
    because the map $F(\barX) \to \AA^1$ is not proper, and if one
    simply compactifies one may add singularities that are worse than nodes.
    The key observations are that 
    \begin{itemize}
    \item Any singularity $N \in \barX_{sing}$ gives a subfamily
      $F(N) \subseteq F(\barX)_{sing}$ that is proper over $\AA^1$.
      Since $\Pi$ is normal, 
      $F(\barX)_\omicron = (\Pi\times \{0\}) \cup (\Lambda \times L\Pone)$
      is nodal, so $\barX$ is at worst nodal.
    \item The new points in the compactification attach to only one component 
      of $F(\barX)_\omicron$, namely $\Pi\times \{0\}$, 
      so aren't relevant in studying connectedness.
    \end{itemize}
    If $N \in \barX_{sing}$, then
    $F(N)_\omicron$ is necessarily one of the nodes $\Lambda \times \{0\}$
    of $(\Pi \times \{0\}) \cup (\Lambda \times L\Pone)$, 
    so as before the formal neighborhood of $F(N)$ inside $F(\barX)$ is a
    trivial deformation of a node. Now compactify, blow up $F(\barX)$
    along $F(N)$, and as before get an impossible family of projective
    curves whose generic fiber is irreducible but whose special fiber
    is disconnected.
  \end{enumerate}
\end{proof}

\subsection{Proofs of the main lemmas}

\begin{proof}[Proof of the Reduced Limit Lemma]
  Via base change, we can reduce to the case that $S$ is the germ of a
  regular $1$-dimensional scheme, e.g. the $\Spec$ of a discrete
  valuation ring $D$. Then $S$ has one closed point and one open point.
  
  By theorem \ref{thm:AK}, the family $F$ is dominated by a family
  $\wtF$ of branchvarieties, agreeing over $S\setminus \omicron$.
  Note that $\wtF_\omicron$, $F_\omicron$ are each equidimensional of
  dimension $d$ and connected in codimension $1$, by proposition
  \ref{prop:degens}.  The branchvariety $\wtF_\omicron \to F_\omicron$
  appears a priori to depend on the curve chosen in the first step, but
  this will not affect the argument (which will in any case
  establish that $\wtF_\omicron = F_\omicron$).

  {\bf The components of $\wtF_\omicron$.}
  Let $\beta_\omicron: \wtF_\omicron \to F_\omicron$ be the induced map on 
  special fibers, 
  with image $(F_\omicron)_{red} = A_1 \union \ldots \union A_k$. 
  We will now show that $\wtF_\omicron$ has exactly the same
  components, though a priori they may be glued together differently.
  
  By the latter conclusion of theorem \ref{thm:AK}, the components
  of $\wtF_\omicron$ can be labeled
  \junk{
    Let $A$ be a component of $\wtF_\omicron$, necessarily of
    dimension $d$.  Since $\beta$ is finite, $\dim \beta(A) = d$.  
    Since $(F_\omicron)_{red}$ is equidimensional of dimension $d$, $\beta(A)$
    is one of its components $A_i$. So far we learn that $\beta$ induces
    a map from components$(\wtF_\omicron) \to$ components$(F_\omicron)$;
    like $\beta$, it is onto.
    
    Now we make use of the assumption that $F_\omicron$ is generically reduced,
    to infer that $\beta$ only takes {\em one} component $A$ of
    $\wtF_\omicron$ to any component $A_i$ of $F_\omicron$, 
    and that the degree of that map $A \to A_i$ is $1$. So the components of 
    $\wtF_\omicron$ can be called 
  }
  $\widetilde A_1,\ldots, \widetilde A_k$, 
  with $\beta_\omicron(\widetilde A_i) = A_i$,
  and each map $\widetilde A_i \to A_i$ is degree $1$ and finite.
  Now we make use of the assumption that the $\{A_i\}$ are normal,
  which lets us infer that each 
  map $\widetilde A_i \to A_i$ 
  is an isomorphism.  So $(F_\omicron)_{red},\wtF_\omicron$
  have the same components, as claimed. Hereafter we identify the
  components of $\wtF_\omicron$ with the $\{A_i\}$.

  {\bf Showing $F_0$ is reduced.}
  We now split into the three cases of the lemma: 
  $k=1$, $k=2$, and general $k$.
  In each case, rather than dealing with $F_\omicron$ directly,
  the idea is to show that the map 
  $\beta_\omicron: \wtF_\omicron \to (F_\omicron)_{red}$ is an isomorphism. 
  Then lemma \ref{lem:crucial} lets us infer indirectly that
  $F_\omicron$ is reduced.

  {\em If $k=1$.} 
  Then $\wtF_\omicron = A_1 = (F_\omicron)_{red}$. 
  By lemma \ref{lem:crucial}, $F_\omicron$ is reduced, and obviously
  normal (since $A_1$ was assumed so).
  
  {\em If $k=2$, and $A_1\cap A_2 \subseteq (F_\omicron)_{red}$ 
    is reduced and irreducible.}
  Consider the diagrams $A_1,A_2 \to (F_\omicron)_{red}$ and
  $A_1,A_2 \to \wtF_\omicron$, and denote their pullbacks
  (which are just the intersections) by $C,C'$. Then by lemma \ref{lem:gluing},
  there is an inclusion $C' \into C$, and our goal is to show their equality.
  Since $\wtF_\omicron$ is connected in
  codimension $1$, $\dim C' = \dim F_\omicron - 1 = \dim C$. 

  Now we use the assumption that $C$ is reduced and irreducible to infer 
  $C' = C$. Hence by lemma \ref{lem:gluing}
  the map $\wtF_\omicron \to (F_\omicron)_{red}$ is
  an isomorphism, so lemma \ref{lem:crucial} tells us $F_\omicron$ is reduced.

  {\em General $k$, $F_\omicron$ reduced through codimension $1$, 
    and the shelling condition.}
  Let $\wtF_\omicron^i, (F_\omicron)_{red}^i$ denote the unions of the
  images of $A_1,\ldots,A_i$ in $\wtF_\omicron, (F_\omicron)_{red}$. 
  So we have a map $\wtF_\omicron^i \onto (F_\omicron)_{red}^i$ 
  for each $i$.
  Assume that $\wtF_\omicron^j \onto (F_\omicron)_{red}^j$
  is an isomorphism for all $j<i$; we will use this to prove it for $j=i$.
  The base case $i=1$ is trivial, as the map is $A_1 \to A_1$.

  This proof is very similar to the one just given, except that
  we work with the pullback diagrams of 
  $\wtF_\omicron^{i-1}, A_i \to (F_\omicron)_{red}^i$ and
  $\wtF_\omicron^{i-1}, A_i \to \wtF_\omicron^i$.
  Again call the pullbacks $C,C'$, and again we have $C' \subseteq C$.
  By assumption $C$ is reduced, 
  and equidimensional of dimension $\dim F_\omicron - 1$.
  So either $C' = C$, or $C'$ does not contain some component $D$ of $C$.
  
  If not, then the map $\wtF_\omicron^i \onto (F_\omicron)_{red}^i$ is
  generically $2$:$1$ over $D$ --- once from $\wtF_\omicron^{i-1}$, 
  once from $A_i$. But then by lemma \ref{lem:crucial},
  $F_\omicron$ would not be generically reduced along $D$.
  
  Hence $C' = C$, so $\wtF_\omicron^i \onto (F_\omicron)_{red}^i$ is
  an isomorphism by lemma \ref{lem:gluing}. When $i=k$, we learn that
  $\wtF_\omicron \onto (F_\omicron)_{red}$ is an isomorphism, so 
  $F_\omicron$ is reduced by lemma \ref{lem:crucial}.

  {\bf Showing the generic fiber is $S_2$.}
  In case (1) we saw that the special fiber is normal, hence $S_2$, 
  so the generic fiber is $S_2$ by proposition \ref{prop:sk}. 
  We now treat cases (2) and (3) together.

  First we show the (reduced!) special fiber $F_\omicron$ is $S_2$.
  This is by induction on $k$, using
  \begin{equation}\label{eqn:Fsi}
    F_\omicron^i = F_\omicron^{i-1} \union_{F_\omicron^{i-1} \cap A_i} A_i 
  \qquad \hbox{for $i>0$, and } F_\omicron^0 := \emptyset
  \end{equation}
  where $F_\omicron^i = \wtF_\omicron^i = (F_\omicron)_{red}^i$. 
  By assumption, ${F_\omicron^{i-1} \cap A_i}$ is reduced so $S_1$, $A_i$ is 
  normal so $S_2$, and $F_\omicron^{i-1}$ is $S_2$ by induction on $i$.
  Their union $F_\omicron^i$ is then $S_2$ by proposition \ref{prop:sk}.
  At $i=k$ we find out $F_\omicron$ is $S_2$.

  Hence by proposition \ref{prop:degens}, the generic fiber is also $S_2$.
  
  {\bf (In cases (1) and (2)) The generic fiber is normal.}
  In case (1), if the abnormal locus in the generic fiber is nonempty,
  its closure will give a subfamily whose special fiber will be nonempty
  and lie in the abnormal locus of $F_\omicron$, contradiction.
  (We could also just invoke the local result of \cite{Hironaka}.)

  Case (2) is exactly the situation of lemma \ref{lem:genfibnormal}
  with either order on the two components,
  and there are no triple intersections to consider.
\end{proof}

In cases (1) and (2) we proved the generic fiber to be normal.
This conclusion need not hold in case (3):
consider a nodal plane cubic degenerating to a union of a 
line and a conic.

Case (2) can be considerably generalized along the lines of the
intersection conditions in lemma \ref{lem:genfibnormal}.
The proof is not any more difficult, but we 
omitted it as we know no natural examples not already covered by case (2).

\begin{proof}[Proof of lemma \ref{lem:skgeneric}]
  Use equation (\ref{eqn:Fsi}) above, proposition \ref{prop:sk}, and induction,
  exactly as was done in the proof above to prove $F_\omicron$ was $S_2$.
\end{proof}

\begin{proof}[Proof of the Geometric Vertex Decomposition Lemma]
  The proof is very close to that of the Reduced Limit Lemma,
  and we use the same notation $S, F_\omicron, \wtF_\omicron$,
  $\beta_\omicron : \wtF_\omicron \to F_\omicron$.

  By theorem \ref{thm:GVDs}, we have the containment of schemes
  $$ F_\omicron \supseteq 
  (\Pi\times \{0\}) \cup_{\Lambda\times \{0\}} (\Lambda\times L\Pone) $$
  and the difference is supported on $\Pi \times \{0\}$, 
  in codimension $\geq 1$.
  (Note that $\Pi$ is automatically irreducible, being the image of $\barX$.)
  In particular $F_\omicron$ is generically reduced.

  {\bf The component $A$.}
  Hence by theorem \ref{thm:AK} there is a component $A$
  of $\wtF_\omicron$ mapping to $\Pi \times \{0\}$, 
  and the map is degree $1$. Having assumed $\Pi$ to be normal,
  we infer that this finite map $A\to \Pi$ is an isomorphism.

  {\bf The union of components $B$.}
  Consider now the preimage 
  $\beta_\omicron^{-1}(F_\omicron \setminus (\Pi \times \{0\})) 
  \subseteq \wtF_\omicron$.
  Since $F_\omicron$ is reduced on 
  $F_\omicron \setminus (\Pi \times \{0\}) 
  = \Lambda \times (L\Pone \setminus \{0\})$
  by the assumption that $\Lambda$ is reduced, using lemma \ref{lem:crucial}
  we can see that the map
  $$ \beta_\omicron: 
  \beta_\omicron^{-1}(F_\omicron \setminus (\Pi \times \{0\})) 
  \to F_\omicron \setminus (\Pi \times \{0\}) $$
  is an isomorphism. 
  Let $B\subseteq \wtF_0$ denote the closure of this preimage.

  Since a component of $\wtF_\omicron$ maps either to 
  $\Pi \times \{0\}$ or to the closure of its complement, we see that 
  $\wtF_\omicron = A \cup B$. The main difference between 
  this situation and case (2) of the Reduced Limit Lemma is that $B$
  is usually not irreducible.

  To continue following the argument in the Reduced Limit Lemma,
  we will need to determine $B$, and show that $A,B$ are glued together
  the same way in $\wtF_\omicron$ as in $(F_\omicron)_{red}$.

  The image of $B$ is the closure of $F_\omicron \setminus (\Pi \times \{0\})$,
  namely $\Lambda \times L\Pone$. 
  So the map $B \to \Lambda \times L\Pone$
  is finite, degree $1$, and an isomorphism away from $\Lambda\times \{0\}$.
  This forces it to be an isomorphism everywhere. 
  (This uses the normality of $L\Pone$ rather than any condition on $\Lambda$.)
  If we use this to identify $B$ with $\Lambda\times L\Pone$, we can decompose
  $$ \wtF_\omicron 
  = (\Pi \times \{0\}) \cup_{C' \times \{0\}} (\Lambda \times L\Pone)
  $$
  where $C'$ is defined by the intersection.

  In $(F_\omicron)_{red}$, namely 
  $(\Pi\times \{0\}) \cup_{\Lambda\times \{0\}} (\Lambda\times L\Pone)$,
  the intersection $(\Pi\times \{0\}) \cap (\Lambda\times L\Pone)$
  is $\Lambda\times \{0\}$ as a scheme. In particular
  this intersection is equidimensional (being Cartier in $\barX$)
  and reduced (by assumption). 
  By lemma \ref{lem:gluing}, $C' \subseteq \Lambda$.
  It remains to show that $C'$ contains general points from each component 
  of $\Lambda$, and thereby learn $C'\times\{0\} = \Lambda\times\{0\}$.
  For this we can safely extend the base field (from the point $\omicron\in S$)
  to its algebraic closure.

  {\bf Slicing down to the $1$-dimensional case.}
  Let $P \subseteq H$ be a plane in general position
  (in particular, not necessarily through $\vec 0$) with respect to
  $\Pi$ and $\Lambda$, whose intersection with $\Lambda$ is $0$-dimensional.
  Then
  \begin{eqnarray*}
  (F_\omicron)_{red} \cap (P\times L\Pone) 
  &=& ((\Pi\cap P) \times \{0\})  
  \cup_{(\Lambda \cap P)\times \{0\}}  ((\Lambda\cap P) \times L\Pone) \\
  \beta_\omicron^{-1}(P\times L\Pone) 
  &=& ((\Pi\cap P) \times \{0\})  
  \cup_{(C' \cap P)\times \{0\}} ((\Lambda\cap P) \times L\Pone)
  \end{eqnarray*}
  Now we make our only use of lemma \ref{lem:bertini}, to say that 
  $\beta_\omicron^{-1}(P\times L\Pone)$ is reduced and that
  $\beta^{-1}(P\times L\Pone \times S)$ is a flat family of branchvarieties.
  Since its generic fiber is irreducible (by Bertini's theorem), 
  proposition \ref{prop:degens} says that 
  $\beta_\omicron^{-1}(P\times L\Pone)$ is connected.

  In $\beta_\omicron^{-1}(P\times L\Pone) 
  = ((\Pi\cap P) \times \{0\})  
  \cup_{(C' \cap P)\times \{0\}} ((\Lambda\cap P) \times L\Pone)$,
  the first term $(\Pi\cap P) \times \{0\}$ is a normal affine curve,
  and $(\Lambda\cap P) \times L\Pone$ is a disjoint union of $\Pone$s.
  For each point $\lambda$ in the finite set $\Lambda\cap P$,
  the only possible point of intersection of $\lambda \times L\Pone$ 
  and any other component of $\wtF_\omicron$ is
  $\lambda \times \{0\}$. 
  So for $\beta_\omicron^{-1}(P\times L\Pone)$ to be connected,
  we must have $\lambda \in C'$. 

  This demonstrates that $C'$ contains general enough points 
  of $\Lambda$ to contain all of $\Lambda$'s top-dimensional components, 
  which with $\Lambda$ reduced says that $C' = \Lambda$.

  Finally, we invoke lemma \ref{lem:gluing} to infer that
  the map $\wtF_\omicron \onto (F_\omicron)_{red}$ is an
  isomorphism, then lemma \ref{lem:crucial} to infer that $F_\omicron$
  is reduced. 

  {\bf Cohen-Macaulayness.}
  We now assume $\Pi$ and $\Lambda$ are Cohen-Macaulay. 
  Then so is $\Lambda \times L\Pone$, so
  $$ (\Pi\times \{0\}) \cup_{\Lambda\times \{0\}} (\Lambda\times L\Pone) $$
  is a union of two Cohen-Macaulay schemes along a third of codimension $1$.
  Hence $F_\omicron$ is Cohen-Macaulay by proposition \ref{prop:sk}.

  We would like to claim that
  $\barX$ is Cohen-Macaulay by proposition \ref{prop:degens}, but that
  proposition assumes projectivity. (Essentially, the problem is that in
  the nonprojective situtation, one can have nonempty families with
  empty special fibers, and we need a different way to forbid this.)
  So instead we consider the 
  non-C-M locus $B_1 \subseteq F_1 = \barX$, and complete it to a 
  subfamily $B \subseteq F$ using the same recipe. Since $B_0$ lies inside
  the (empty) non-C-M locus of $F_0$, it too is empty.
  By lemma \ref{lem:gvdslice}, 
  since $B_0 = \emptyset$, then $B_1 = \emptyset$ also.
  
  Finally, $X$ is Cohen-Macaulay since it is open in $\barX$.
\end{proof}

\begin{proof}[Proof of lemma \ref{lem:gvdnormal}]
  Let $C_1 \subseteq \barX_{sing}$ be a component of the 
  singular locus, and of codimension $1$ in $\barX$.
  Let $C \subseteq F$ be the subfamily constructed by the same recipe as $F$. 
  Then $C_t \subseteq (F_t)_{sing}$ for all $t \in \AA^1$.

  {\bf Showing $C_1 \not\subseteq \Lambda \times \{\infty\}$.}
  Assume for contradiction that $C_1 \subseteq \Lambda \times \{\infty\}$.
  Then $C$ is a constant family, so 
  $C_0 \subseteq (\Lambda \times \{\infty\}) \cap (F_0)_{sing}$.
  By theorem \ref{thm:GVDs}
  $$ F_0 \setminus (H\times \{0\}) 
  = \Lambda \times (L\Pone \setminus \{0\}) $$
  so
  $$ (F_0 \setminus (H\times \{0\}))_{sing} 
  = (\Lambda \times (L\Pone \setminus \{0\}))_{sing}
  = \Lambda_{sing} \times (L\Pone \setminus \{0\}). $$
  Hence 
  $C_0 \subseteq \Lambda_{sing} \times \{\infty\}$.

  Since $\Lambda$ is generically reduced
  (indeed, it was assumed reduced), its singular locus $\Lambda_{sing}$ 
  contains no top-dimensional components of $\Lambda$. 
  But then its dimension is too small to contain $C_0$, contradiction.
  
  {\bf Showing $C_1 = D\times L\Pone$.}
  \junk{
    Having shown that $C_1 \not\subseteq \Lambda \times \{\infty\}$,
    we may apply theorem \ref{thm:GVDs} to $C$, with
    $$ C_\omicron = 
    (\Pi_C\times \{0\}) \cup_{\Lambda_C\times \{0\}} (\Lambda_C\times L\Pone)
    \qquad \hbox{as sets} $$
    where $\Pi_C$ is the image of the projection of $C_1$ and
    $\Lambda_C \times \{\infty\} = C_1 \cap (\Lambda \times \{\infty\})$.
  }
  Extend the base field as usual, 
  so we may slice with a general plane $P\subseteq H$ such that 
  \begin{itemize}
  \item $P\cap \Pi$ is a normal curve, and
  \item $P\cap \barX$ is irreducible 
    and projects generically $1:1$ to $P\cap \Pi$, and
  \item $P\cap \Lambda$ is reduced and $0$-dimensional.
  \end{itemize}
  Then
  \begin{eqnarray*}
    F_\omicron \cap (P\times L\Pone) 
    &=& \left(
      (\Pi\times \{0\}) 
      \cup_{\Lambda\times \{0\}} (\Lambda\times L\Pone)\right) 
    \cap (P\times L\Pone) \\ 
    & = & \left( (\Pi\cap P) \times \{0\}\right) 
    \cup_{(\Lambda\cap P)\times \{0\}} 
    \left((\Lambda \cap P)\times L\Pone\right)
  \end{eqnarray*}
  is nodal, so
  we can apply lemma \ref{lem:gvdslice} part (2) to infer 
  that $\barX \cap (P\times L\Pone)$ is a normal curve.

  At this point we assume, for intended contradiction, that $C$ is
  {\em not} of the form $D\times L\Pone$. 
  Therefore its projection $\Pi_C \subseteq H$ is of the same dimension as $C$,
  so $\Pi_C \cap P$ is a nonempty set of points.
  Consequently, $C \cap (P \times L\Pone)$ is a nonempty set of
  singularities of the curve $\barX \cap (P \times L\Pone)$,
  contradiction.

  {\bf Studying $D$.}
  What can we say about the factor $D$ in $C_1 = D\times L\Pone$? Since
  $$ \Lambda \times \{\infty\} 
  = \barX \cap (H\times \{\infty\})
  \supseteq C_1 \cap (H\times \{\infty\})
  = (D\times L\Pone) \cap (H\times \{\infty\})
  = D\times \{\infty\},$$ 
  we see $\Lambda \supseteq D$.
  By dimension count, $D$ is codimension $1$ in $\Lambda$,
  and $C_1 \subseteq \barX_{sing}$ implies $D \subseteq \Lambda_{sing}$,
  whose codimension in $\Lambda$ is at least $1$ (since $\Lambda$ was 
  assumed reduced).
  Hence $D$ is a top-dimensional (in particular, non-embedded) component
  of $\Lambda_{sing}$.

  If $\Lambda$ is normal, then it is $R_1$ so there can be no such $D$
  (since $\Lambda_{sing}$ is of too low dimension)
  thus no such $C_1$, hence $\barX$ also is $R_1$. 
  Also, by the same Mayer-Vietoris argument as in
  the proof of Cohen-Macaulayness in the previous lemma, $\Lambda$ and $\Pi$ 
  being $S_2$ (since they are normal) implies that $\barX$ is $S_2$.
  Together, we see that $\Lambda$ normal implies $\barX$ is normal.
\end{proof}

We now give an example showing the criterion in 
lemma \ref{lem:gvdnormal} is not automatic. Let
\begin{eqnarray*}
  \Pi = H &=& \Spec \complexes[x,y] \\
  L\Pone &=& \Proj \complexes[a^{(1)},b^{(1)}] \qquad
  \text{(superscripts indicating degrees)} \\
  \barX 
  &=& \Proj \complexes[x^{(0)},y^{(0)},a^{(1)},b^{(1)}]
  \big/ \<a x^2 - (a+b) y^2\>
\end{eqnarray*}
so 
$$ \Lambda = \left\{ (x,y,[a,b]) \in \barX : b=0 \right\} 
  \iso \left\{(x,y) : x = \pm y\right\} $$
Then all the other conditions hold:  $\barX$ is irreducible,
$\barX \to \Pi$ is generically $1$:$1$, $\Pi$ is normal, $\Lambda$ is $S_2$,
but $\barX$ is not $R_1$. 
According to lemma \ref{lem:gvdnormal}, 
we may blame this on the $\Pone$ of singularities along $\{x=y=0\}$,
which happens to be the support of the whole singular locus of $\barX$.

\section{An application to Schubert varieties}\label{sec:Schubert}

In this section we use the Geometric Vertex Decomposition Lemma 
and lemma \ref{lem:gvdnormal} to study the 
singularities of Schubert varieties in generalized flag manifolds $G/B$.

Most of the results here are standard, or at least well-known to
the experts, with the exception of lemma \ref{lem:conecomponent}
and of course the new proof of theorem \ref{thm:schubertCM}.
Since our goal is exactly to provide a new proof, we felt it was
worth making the argument largely self-contained, the better to 
demonstrate that we haven't hidden an old proof somewhere.
The exceptions to self-containment are some structure theory
of reductive groups and the BGG/Demazure iterative construction
of Schubert varieties.

\subsection{Copies of $H\times L\Pone$ inside $G/B$}

Fix a pinning $(G,T,W,B,B_-)$ of a reductive Lie group, 
where $B$ and $B_-$ are opposed Borel subgroups with intersection $T$.
Let $N_-$ denote the unipotent radical of $B_-$, so $B_- = T N_-$.
Associated to a simple root $\alpha$
we have the simple reflection $r_\alpha \in W := N(T)/T$ and the
minimal parabolic subgroup $P_\alpha = \< B, r_\alpha B r_\alpha \>$.
Let $P_{-\alpha} = \< B_-, r_\alpha B_- r_\alpha \>$
denote the corresponding extension of $B_-$.

For example, $G$ might be the group $GL_n$ 
of invertible matrices,
$B$ the upper triangular matrices, 
$P_\alpha$ the matrices whose lower triangle vanishes except at
the matrix entry $(j+1,j)$,
$B_-$ the lower triangulars, and $T$ the diagonals. 
Our interest is in the {\dfn generalized flag manifold} $G/B$,
which is isomorphic in the case $G=GL_n$ 
to the space of full flags in the vector space $\AA^n$. 
In general, since our interest is in the action on $G/B$, there is no harm in
replacing $G$ by its adjoint group $G/Z(G)$.
Let $\pi_\alpha: G/B \onto G/P_\alpha$ denote the canonical submersion,
a bundle map with fibers $P_\alpha/B \iso \Pone$.
\junk{(All our schemes and flat families will be defined over $\integers$,
  but it is convenient to work over a field, in order 
  to define Hilbert functions of our coordinate rings.)}

It turns out that $G/B$ is an especially natural venue in which to
apply the Geometric Vertex Decomposition Lemma: it contains many open
subvarieties of the form $H\times L\Pone$. To locate them we will need
a couple of preparatory statements.

\begin{Lemma}\label{lem:Nmult}
  \begin{enumerate}
  \item   Let $X \subseteq N_-$ be closed, $T$-invariant, and nonempty.
    Then $X \ni 1$. 
  \item 
    Let $v \in W$, 
    and $N_1 = N_- \cap v N_- v^{-1}, N_2 = N_- \cap v N v^{-1}$.
    Then the multiplication maps $N_1 \times N_2 \to N_-$,
    $N_2 \times N_1 \to N_-$ are isomorphisms of schemes. 
  \end{enumerate}
\end{Lemma}

\begin{proof}
  If $\sigma: \Gm \to T$ is a regular dominant
  coweight, then 
  $$ \forall n\in N_-,\quad \lim_{t\to 0}\ {\rm ad\ }\sigma(t) \cdot n = 1.$$
  So $\exists n\in X$ plus $X$ invariant under $T$ (hence under $\sigma$) 
  implies $1\in X$.

  We consider the first map in the second claim 
  (the argument is the same for the second map).
  This map is $T$-equivariant with respect to the
  conjugation action of $T$ on $N_-,N_1,N_2$.
  Hence the semialgebraic sets
  $$ \{ (n_1,n_2) \in N_1\times N_2 : \exists (n_1',n_2') \neq (n_1,n_2),
  n_1' n_2' = n_1 n_2 \} $$
  $$ \{n \in N_- : \not\!\exists (n_1,n_2)\in N_1\times N_2, n_1 n_2 = n\} $$
  $$ \text{the ramification locus in $N_-$} $$
  are each $T$-invariant, and we can apply the first claim to their closures.
  
  The derivative at the identity of this map is the 
  isomorphism $\lie{n}_1 \oplus \lie{n}_2 \to \lie{n_-}$.
  By the inverse function theorem, 
  the multiplication map is a diffeomorphism near the identity of $N_1,N_2$.
  So these sets cannot have the identity in their closure, and hence
  must be empty. Therefore this map is a bijective, unramified map 
  between normal varieties, hence an isomorphism.  
\end{proof}

Since the map $\pi_\alpha: G/B \onto G/P_\alpha$ is a fiber bundle,
it can be trivialized over some atlas of the target, and the
following proposition specifies one of local trivializations.
We will need the fact that $N_-$ acts on $G/B$ with a free open dense
orbit, called the {\dfn big cell}.

\newcommand\barF{\overline F}
\begin{Proposition}\label{prop:param}
  Let $Rad(P_{-\alpha}) := N_- \cap r_\alpha N_- r_\alpha$, 
  the unipotent radical of $P_{-\alpha}$. 
  Pick a group isomorphism 
  $F_{-\alpha} : \Ga \to N_{-\alpha} := N_- \cap r_\alpha N r_\alpha$.
  Then the map 
  $$ \AA^1 \to G/B, \quad z \mapsto F_{-\alpha}(z) B / B $$
  extends continuously to an embedding $\barF_{-\alpha}: \Pone \to G/B$, 
  and the $T$-equivariant map  
  \begin{eqnarray*}
    \gamma: Rad(P_{-\alpha}) \times \Pone &\to& G/B \\
    (n, z) &\mapsto& n F_{-\alpha}(z) B/B
  \end{eqnarray*}
  is an open immersion, with image $\pi_\alpha^{-1}(N_- P_\alpha/P_\alpha)$.
  The diagram of $\Pone$-bundles 
$$
  \begin{array}{ccccccc}
    Rad(P_{-\alpha}) \times \Pone 
    & \stackrel{ \gamma}\longrightarrow & \ \ \pi_\alpha^{-1}(N_- P_\alpha/P_\alpha) &\subseteq& G/B \\
    \downarrow & & \ \ \ \ \ \ \ \downarrow \pi_\alpha  &&\downarrow \pi_\alpha\\
    Rad(P_{-\alpha}) &  \stackrel{\cdot P_\alpha/P_\alpha}\longrightarrow & \ \ N_- P_\alpha/P_\alpha &\subseteq& G/P_\alpha
  \end{array}
$$
  commutes, and the horizontal arrows are isomorphisms,
  making $\gamma$ a trivialization of the $\Pone$-bundle $\pi_\alpha$
  restricted to the open set $N_- P_\alpha/P_\alpha \subseteq G/P_\alpha$.

  Let $X_{r_\alpha}^\circ := N_- r_\alpha B/B \subseteq G/B$. 
  Then the locally closed subsets
  $$ \gamma\left(Rad(P_{-\alpha}) \times \{\infty\}\right) 
  = X_{r_\alpha}^\circ, \qquad
  \gamma\left(Rad(P_{-\alpha}) \times \{0\}\right) 
  = r_\alpha X_{r_\alpha}^\circ $$
  are sections over $N_- P_\alpha/P_\alpha$ of this $\Pone$-bundle.
\end{Proposition}

\begin{proof}
  The restriction of $\gamma$ to the open set
  $$ Rad(P_{-\alpha}) \times \AA^1 \to G/B $$
  factors through $N_-$ by lemma \ref{lem:Nmult}, 
  hence that restriction is an open immersion onto the big cell.
  We must show that the extension exists and is finite (e.g. injective), with 
  image a normal variety, to conclude that it is an isomorphism onto the image.
  Being an extension of an injective immersion, it is automatically
  degree $1$.

  The extension of 
  $$ \AA^1 \to G/B, \quad z \mapsto F_{-\alpha}(z) B / B $$
  to $\Pone$ exists because $G/B$ is proper.
  Since the group $F_{-\alpha}$ is contained in $P_\alpha$, 
  the image of $\barF_{-\alpha}$ is contained in the fiber $P_\alpha/B$
  of the map $\pi_\alpha$, and as they are both $1$-dimensional, closed,
  and reduced the image must equal that fiber. 
  Also, this degree $1$ proper map to a normal target must be an isomorphism.
  We mention that the previously missed point in $P_\alpha/B$ is $r_\alpha B$.

  By lemma \ref{lem:Nmult} applied to $v=r_\alpha$, 
  the intersection $Rad(P_{-\alpha}) \cap F_{-\alpha}$ is trivial.
  Hence the orbit through the basepoint of $G/P_\alpha$ is free,
  and by dimension count, open dense. 
  Call this orbit the {\dfn big cell on $G/P_\alpha$}. 
  
  We now claim that $\gamma$ is an
  isomorphism of $Rad(P_{-\alpha}) \times \Pone$ and 
  $\pi_\alpha^{-1}(N_- P_\alpha/P_\alpha)$.
  Since $\pi_\alpha$ is $G$-equivariant,
  each element $n \in Rad(P_{-\alpha}) \leq G$ permutes the fibers.
  Because $Rad(P_\alpha)$ acts freely on the big cell on $G/P_\alpha$, 
  it doesn't preserve any fiber, which shows that $\gamma$ is injective.

  The image of $\gamma$ is obviously a union of fibers, 
  and composing with $\pi_\alpha$
  the map becomes $(n,z) \mapsto n P_\alpha B/B$, 
  whose image doesn't change if we replace $n \in Rad(P_{-\alpha})$ 
  by $n \in Rad(P_{-\alpha}) F_{-\alpha} = N_-$. Hence the image
  of $\gamma$ is $\pi_\alpha^{-1}(N_- P_\alpha/P_\alpha)$ as claimed.

  Since $\pi_\alpha \circ \gamma$ is proper, so is $\gamma$, 
  hence it is a proper bijective degree $1$ map to its normal image, 
  and thus an isomorphism.

  Obviously $Rad(P_{-\alpha}) \times \{z\}$ is a section of the
  left-hand bundle for $z = 0,\infty$ or indeed any $z \in \Pone$.
  We compute the images, using
  $N_\alpha = N \cap r_\alpha N_- r_\alpha$:
  \begin{eqnarray*}
    \gamma\left(Rad(P_{-\alpha}) \times \{\infty\}\right) 
    &=& Rad(P_{-\alpha}) r_\alpha B/B
    = Rad(P_{-\alpha}) r_\alpha N_\alpha B/B
    = Rad(P_{-\alpha}) N_{-\alpha} r_\alpha B/B \\
    &=& N_- r_\alpha B/B
    = X_{r_\alpha}^\circ \\
    \gamma\left(Rad(P_{-\alpha}) \times \{0\}\right) 
    &=& Rad(P_{-\alpha}) B/B
    = r_\alpha Rad(P_{-\alpha}) r_\alpha B/B
    = r_\alpha X_{r_\alpha}^\circ. 
  \end{eqnarray*}
\end{proof}

The image of $\barF_{-\alpha}$ is a $T$-invariant $\Pone$ inside $G/B$,
whose $T$-fixed points are $\{B, r_\alpha B\}$. One consequence of 
proposition \ref{prop:param} is that this $\Pone$ has trivial normal bundle
(and enjoys a tubular neighborhood theorem, a rarity in algebraic geometry).
This triviality does {\em not} hold on partial flag manifolds $G/P$; 
for example the normal bundle to a $T$-invariant $\Pone \subset \PP^2$
is $\mathcal O(1)$, not trivial. This is the uncommon situation
in which $G/B$ is simpler than $G/P$.

\subsection{Schubert varieties and patches}

The {\dfn Chevalley-Bruhat decomposition} of the 
{\dfn generalized flag manifold} $G/B$ is by orbits of $N_-$,
which are indexed by the Weyl group $W$:
$$ G/B = \coprod_{w\in W} X_w^\circ, 
\qquad\qquad
X_w^\circ  := N_- w B / B. $$
(Technically, $w \in W = N_G(T)/T$ should be lifted to an element
$\widetilde w \in N_G(T)$, but $\widetilde w B$ doesn't depend on
this choice, so we don't clutter the notation with it.)
A {\dfn Schubert variety} $X_w$ is 
the closure $\overline{X_w^\circ} \subseteq G/B$ 
of a {\dfn Schubert cell} $X_w^\circ$.
The big cell is $X_1^\circ$.
The {\dfn Bruhat order} on $W$ has $v\geq w$ if $vB \in X_w$. 
Then each Schubert variety has a Bruhat
decomposition $X_w = \coprod_{v\geq w} X_v^\circ$.

\junk{
  Since $G/B$ is irreducible and $N_-$ is connected, 
  there must be exactly one open $N_-$-orbit
  in this finite decomposition, and it is $N_- B/B$.
  This is a copy of affine space, called the {\dfn big cell},
  though some authors prefer to call the open $B$-orbit 
  $B w_0 B/B = w_0 N_- B$ the big cell,
  where $w_0\in W$ is the {\dfn long element}.
}

We summarize what we need of the Bernstein-Gel$'$fand-Gel$'$fand/Demazure 
iterative construction of Schubert varieties.

\begin{Proposition}\label{prop:BGG}
  Let $\pi_\alpha$ denote the canonical submersion $G/B \onto G/P_\alpha$.
  Then $\pi_\alpha^{-1}(\pi_\alpha(X_v))$, being manifestly irreducible,
  closed, and $B_-$-invariant must again be a Schubert variety. 

  There are two cases: if
  $v r_{\alpha} > v$, then $\pi^{-1}(\pi(X_v)) = X_v$ again,
  whereas if
  $v r_{\alpha} < v$, then $\pi^{-1}(\pi(X_v)) = X_{v r_{\alpha}}$.
  In the latter case, the map $\pi_\alpha: X_v \onto \pi_\alpha(X_v)$
  is generically $1$:$1$.
\end{Proposition}

Define a {\dfn Schubert patch} $X_w|_v$ as the intersection
$$ X_w|_v := X_w \ \cap\  (v N_- B/B). $$
Since $N_- B/B$ is a copy of affine space, the Schubert patch naturally sits 
inside it as an affine subvariety, and both carry an action of $T$. 
It will also be useful to have the notation
$$ X_w|_S := \Union_{s\in S} X_w|_s \qquad\text{for any subset }S\subseteq W.$$
Schubert patches have been studied before, most obviously in 
Kazhdan-Lusztig theory; we include some more relevant 
refences in section \ref{ssec:subword}.

\begin{Proposition}\label{prop:cover}
  \begin{enumerate}
  \item $X_v^\circ \subseteq vN_- B/B$ for each $v\in W$.
  \item $X_w|_v$ is nonempty iff $v\geq w$.
  \item The Schubert patches $\{X_w|_v\}_{v\geq w}$ on a Schubert
    variety $X_w$ are an open cover. In particular, $X_w$ is normal
    and Cohen-Macaulay iff each $X_w|_v$ is.
  \item The Schubert patch $X_v|_v$ is just the Schubert cell $X_v^\circ$.
  \item \label{prop:cover:missed}
    If $u\not\leq v$, then $X_w|_v \cap X_u = \emptyset$.
  \item \label{prop:cover:image}
    The image of the open immersion 
    $\gamma : Rad(P_{-\alpha}) \times \Pone \to G/B$
    from proposition \ref{prop:param} is 
    $X_1^\circ \coprod X_{r_\alpha}^\circ = X_1|_{1,r_\alpha}$.
  \end{enumerate}
\end{Proposition}

\begin{proof}
  \begin{enumerate}
  \item Fix $v$, and let $N_1,N_2$ be as in lemma \ref{lem:Nmult}.
    Then
    $$ N_- v B/B = N_1 N_2 v B/B = v (v^{-1} N_1 v) (v^{-1} N_2 v) B/B. $$
    Now since 
    $$ v^{-1} N_2 v = v^{-1} N_- v \cap N  \subseteq N \subseteq B, \qquad
    v^{-1} N_1 v  = v^{-1} N_- v \cap N_- \subseteq N_-, $$
    we find
    $$ N_- v B/B = v (v^{-1} N_1 v) B/B \subseteq v N_- B/B. $$
  \item The intersection $X_w \cap v N_-B/B$ is a closed $T$-invariant
    subset of $v N_-B/B$. By lemma \ref{lem:Nmult}, it is only nonempty
    if it contains the basepoint $v B$, i.e. if $vB \in X_w$, 
    or equivalently $v \geq w$.
  \item By part (1),
    $ \Union_{v\in W} v N_- B/B \supseteq \Union_{v\in W} N_- v/B = G/B$.
    Intersecting with $X_w$, we get $\{X_w|_v\}$ as an open cover on $X_w$.
    Since normality and Cohen-Macaulayness are local conditions, it is
    enough to check them on an open cover.
  \item 
    Use $X_w = \coprod_{v\geq w} X_v^\circ$ to write
    $$ X_w|_w = X_w \cap w N_- B/B 
    = (X_w^\circ \cap w N_-B/B) \union \coprod_{v>w}
    (X_v^\circ \cap w N_-B/B). $$
    Since $w N_- B/B \supseteq N_- w B/B = X_w^\circ$,
    the first term is the desired one, and it remains to show that 
    $X_v^\circ \cap w N_-B/B = \emptyset$ for $v>w$. 
    
    The set $w^{-1} X_v \cap N_- B/B$ is a $T$-invariant closed subset
    of the big cell, so by lemma \ref{lem:Nmult} 
    if it is nonempty it must contain the basepoint, i.e. $wB/B \in X_v$.
    But this is only true for $w\geq v$, contradicted by $v>w$. 
  \item Since $X_w|_v = X_w \cap X_1|_v$, it is enough to prove this 
    for $w=1$, using (2).
    \junk{
      Then $X_1|_v \cap X_u$ is a $T$-invariant closed subset
      of $X_1|_v$, and by lemma \ref{lem:Nmult} is only nonempty if
      it contains the basepoint $vB$, i.e. if $vB \in X_u$.
    }
  \item 
    We already computed the image to be 
    $\pi_\alpha^{-1}(N_- P_\alpha/P_\alpha)$ so it is plainly 
    $N_-$-invariant. Hence it is determined by which $wB$ it contains,
    $w\in W$, and in this case the only such $w$ are $\{1,r_\alpha\}$.
    So the $N_-$-orbit decomposition of the image is
    \begin{eqnarray*}
      X_1^\circ \coprod X_{r_\alpha}^\circ 
      &=& X_1|_1\coprod X_{r_\alpha}^\circ\\
      &\subseteq& X_1|_1 \union {r_\alpha} X_1|_1 = X_1|_{1,r_\alpha}
    \end{eqnarray*}
    which gives us one of the desired inclusions.

    For the opposite inclusion, we need to show 
    $ \pi_\alpha^{-1}(N_- P_\alpha/P_\alpha) \supseteq {r_\alpha} X_1|_1$.
    Note that
    $$ N_- P_\alpha = B_- P_\alpha = B_- (SL_2)_\alpha P_\alpha 
    = P_{-\alpha} P_\alpha = (SL_2)_\alpha B_- P_\alpha $$
    where $(SL_2)_\alpha \subseteq P_\alpha$ denotes the root $SL_2$ subgroup.
    Hence $N_- P_\alpha/P_\alpha$ is $r_\alpha$-invariant, and so is its 
    $\pi_\alpha$ preimage. 
    Therefore that preimage contains ${r_\alpha} X_1|_1$.
  \end{enumerate}
\end{proof}

We will prove that these affine patches are normal and Cohen-Macaulay
by induction on $v$ with respect to the Bruhat order.
Their study will require the following technical lemma, whose proof
was sketched for us by Shrawan Kumar.

\begin{Lemma}\label{lem:conecomponent}
  Let $r_\alpha$ be a simple reflection, and $w,v \in W$ such that
  $v r_\alpha < v$. 
  Then 
  $$ \left( X_w \cap v\cdot X_{r_\alpha}^\circ \right)
  \times \AA^1_{v\cdot\alpha} \, \iso\, X_w|_{v r_\alpha}. $$ 
\end{Lemma}

\begin{proof}
  Let $N_{v\cdot \alpha} := v (N \cap r_\alpha N_- r_\alpha) v^{-1}$
  denote the $T$-invariant one-parameter subgroup of $N_-$
  with $T$-weight $v\cdot\alpha$. 
  Similarly, let $N_{-\alpha} = r_\alpha N r_\alpha \cap N_-$.
  We will prove that the multiplication map
  $$
  N_{v\cdot\alpha} \times (X_w \cap v\cdot X_{r_\alpha}^\circ) 
  \to X_w|_{v r_\alpha}
  $$
  or equivalently
  $$
  N_{v\cdot\alpha} \times (X_w \cap v N_- r_\alpha B/B) 
  \to X_w \cap v r_\alpha N_-B/B 
  $$
  is the desired isomorphism of schemes. Acting by $(v r_\alpha)^{-1}$, 
  we may instead study
  $$ N_{-\alpha} \times \left(
  (r_\alpha v^{-1} \cdot X_w) \cap r_\alpha N_- r_\alpha B/B  \right)
  \to (r_\alpha v^{-1} \cdot X_w) \cap N_- B/B. $$
  First we confirm that this map takes values in the space claimed.
  Since $N_- r_\alpha B/B \subseteq r_\alpha N_- B/B$ by
  proposition \ref{prop:cover}, we see
  $r_\alpha N_- r_\alpha B/B \subseteq N_- B/B$ and
  $$ (r_\alpha v^{-1} \cdot X_w) \cap r_\alpha N_- r_\alpha B/B 
  \subseteq (r_\alpha v^{-1} \cdot X_w) \cap N_- B/B. $$
  Therefore it is enough to show that the target space is
  $N_{-\alpha}$-invariant. 
  Obviously $N_- B/B$ is invariant under $N_-$, hence under $N_{-\alpha}$.
  And $r_\alpha v^{-1} \cdot X_w$ is $N_{-\alpha}$-invariant iff
  $X_w$ is $N_{v\cdot\alpha}$-invariant, which it is since $v\cdot\alpha$
  is a negative root. (This is where we use $v r_\alpha < v$.)

  In the $w=1$ case, this map
  $$ N_{-\alpha} \times r_\alpha N_- r_\alpha B/B \to N_- B/B \iso N_- $$
  is the $v=r_\alpha$ case of the (latter) map in lemma \ref{lem:Nmult}, 
  hence an isomorphism. Since this map takes each intersection
  with $r_\alpha v^{-1} \cdot X_w$ into itself, each restriction
  $$ N_{-\alpha} \times 
  (r_\alpha v^{-1} \cdot X_w) \cap r_\alpha N_- r_\alpha B/B 
  \to (r_\alpha v^{-1} \cdot X_w) \cap N_- B/B $$
  of this map is also an isomorphism.
\end{proof}

There are several possibilities for the relative positions of
$v,vr_\alpha,w,wr_\alpha$; only one case will need the full strength of 
the Geometric Vertex Decomposition Lemma. We start with the cases
that don't.

\begin{Proposition}\label{prop:easycases}
  Let $w,v \in W$, and assume $v\geq w$. Fix a simple root $\alpha$
  and the submersion $\pi_\alpha : G/B \onto G/P_\alpha$.
  In the following, $\AA^1_\beta$ denotes the $1$-dimensional 
  representation of $T$ with weight $\beta$ and
  $\Pone_\beta$ its projective completion, and all isomorphisms
  claimed are $T$-equivariant.

  1. Assume $w < w r_\alpha$. Then $v r_\alpha \geq w$, and
  \begin{eqnarray*}
    X_w|_v 
    &\iso& \pi_\alpha(X_w|_{v,vr_\alpha}) \times \AA^1_{-v\cdot\alpha} \\
    X_w|_{v r_\alpha} 
    &\iso& \pi_\alpha(X_w|_{v,vr_\alpha}) \times \AA^1_{v\cdot\alpha}
  \end{eqnarray*}
  
  2. If $v r_\alpha \not \leq w$, 
  then $w > w r_\alpha$ and $v > v r_\alpha$. If in addition
  $X_{wr_\alpha}|_{v,vr_\alpha}$ is normal (as indeed it is), then
  $$ X_w|_v \times \Pone_{-v\cdot\alpha} \iso X_{wr_\alpha}|_{v,vr_\alpha}.$$

  3. If $w r_\alpha > w$ for all simple roots $\alpha$, then $w=1$
  and $X_w|_v$ is smooth for all $v$.
\end{Proposition}

\begin{proof}
  For any $u\in W$ we have $uP_\alpha = ur_\alpha P_\alpha$,
  so the points $uB, ur_\alpha B$ lie in the same
  $\Pone$ fiber of $\pi_\alpha$.
  If $w < w r_\alpha$, then $\pi_\alpha: X_w \to \pi_\alpha(X_w)$
  is a $\Pone$-bundle, so $u \geq w \iff u r_\alpha \geq w$.

  1. Since we're assuming $w < wr_\alpha$ and $w \leq v$, by the above we
  have $w \leq v r_\alpha$ too. 
  
  Since $w < w r_\alpha$, by the BGG/Demazure proposition \ref{prop:BGG}
  we know $X_w$ is a union of fibers of $\pi_\alpha$.  In particular, 
  restricting the bundle $\pi_\alpha$ to 
  the base $v\cdot N_- P_\alpha/P_\alpha$ we see
  $$ 
  X_w|_{v,vr_\alpha}\ \iso\ 
  \pi_\alpha(X_w|_{v,vr_\alpha}) \times \Pone_{-v\cdot\alpha}
  \qquad \text{using $\gamma$ from proposition \ref{prop:param}.}
  $$
  Twist proposition \ref{prop:cover} part (\ref{prop:cover:image})
  by $v,vr_\alpha$:
  $$
  X_1|_{v,vr_\alpha} = X_1|_v \coprod v\cdot X_{r_\alpha}^\circ,
  \qquad
  X_1|_{v,vr_\alpha} 
  = X_1|_{vr_\alpha} \coprod v r_\alpha \cdot X_{r_\alpha}^\circ 
  $$
  Intersect with $X_w$ and rewrite:
  $$
  X_w|_v = X_w|_{v,vr_\alpha} \setminus (v\cdot X_{r_\alpha}^\circ), \qquad
  X_w|_{v r_\alpha}
  = X_w|_{v,vr_\alpha} \setminus (v r_\alpha \cdot X_{r_\alpha}^\circ)
  $$
  By proposition \ref{prop:param}, $v X_{r_\alpha}^\circ$ and 
  $v r_\alpha X_{r_\alpha}^\circ$ correspond under $\gamma$
  to the $\infty$ and $0$ sections. So the isomorphism above restricts to
  $$ 
  X_w|_v          \ 
  \iso\ \pi_\alpha(X_w|_{v,vr_\alpha}) 
  \times (\Pone_{-v\cdot\alpha}\setminus \{\infty\}), 
  \qquad
  X_w|_{vr_\alpha}\ 
  \iso\ \pi_\alpha(X_w|_{v,vr_\alpha}) 
  \times (\Pone_{-v\cdot\alpha}\setminus \{0\})
  $$
  as desired. 

  2. Since $v \in X_w$, if $w < wr_\alpha$ then $v r_\alpha \in X_w$, 
  contradiction.
  
  If $v < v r_\alpha$, then $w \leq v$ implies $w < v r_\alpha$,
  again a contradiction. 
  
  Consider the map 
  $$ \pi: X_w|_v \to \pi_\alpha(X_{wr_\alpha}|_{v,vr_\alpha}) $$
  restricted from $\pi_\alpha$.
  By the BGG/Demazure proposition \ref{prop:BGG}, it is onto and 
  generically $1$:$1$, which is where we use the assumption $w r_\alpha < w$.
  Since $X_{wr_\alpha}|_{v,vr_\alpha}$ is assumed normal, 
  and is a (trivial) $\Pone$-bundle over the target, 
  the target is also normal. Consider the open subset 
  $$ U = \{ u \in \pi_\alpha(X_{wr_\alpha}|_{v,vr_\alpha}) 
  : \pi^{-1}(u) \text{ is finite} \}, $$
  which is also normal. Then the map $\pi : \pi^{-1}(U) \onto U$ is proper 
  (being the restriction of the proper map $X_w \onto \pi_\alpha(X_w)$),
  finite by construction, and generically $1$:$1$; 
  since its target $U$ is normal (being open 
  in the normal variety $\pi_\alpha(X_{wr_\alpha}|_{v,vr_\alpha})$)
  we see $\pi : \pi^{-1}(U) \onto U$ is an isomorphism.
  
  What this shows is that the fibers of $\pi$ that are not single
  (reduced) points are entire $\Pone$s 
  (the fibers of $\pi_\alpha: G/B \onto G/P_\alpha$). 
  We now wish to show that no such fibers occur, i.e. $\pi$ is an isomorphism.

  Let $\pi_\alpha|_{X_w}: X_w \to G/P_\alpha$ be 
  the restriction of $\pi_\alpha$. Since it is $N_-$-equivariant, 
  over each $N_-$-orbit in $G/P_\alpha$ the fiber is constant. 
  By the BGG/Demazure proposition \ref{prop:BGG}, the $\pi_\alpha$-preimages 
  of those orbits are $X_u^\circ \coprod X_{u r_\alpha}^\circ$ 
  for $u < u r_\alpha$.

  Intersecting with $X_w|_v$ to get the $\pi$-preimages, 
  and using proposition \ref{prop:cover} part (\ref{prop:cover:missed}),
  the $\pi$-preimage of $\pi_\alpha(X_u^\circ)$ is empty unless $u \leq v$.

  Let $Q = \{gB \in G/B : X_w \supseteq \pi_\alpha^{-1}(gP_\alpha)\}$,
  the $N_-$-invariant closed set of big $\pi_\alpha|_{X_w}$ fibers.
  Being $N_-$-invariant, it has a Bruhat decomposition,
  $$ Q = \coprod_{u\in W\ :\ uB\in Q} X_u^\circ; $$
  being closed, its $N_-$-orbit set $\{u\in W : uB\in Q\}$ is closed under
  going up the Bruhat order.

  Now we use the assumption $v r_\alpha \not\geq w$, to see $v \notin Q$. 
  Hence $u\in Q \implies u\not\leq v$. 

  Hence $\pi : X_w|_v \to \pi_\alpha(X_{wr_\alpha}|_{v,vr_\alpha})$
  has no $\Pone$-fibers, so is a finite degree $1$ proper map, 
  hence an isomorphism.

  Now we use the fact from proposition \ref{prop:param} 
  that $\pi_\alpha$ is a trivial $\Pone$-bundle over
  $\pi_\alpha(X_{wr_\alpha}|_{v,vr_\alpha})$ to derive the desired
  isomorphism.
\end{proof}

\junk{
  The remaining possibility, $v > v r_\alpha \geq w > wr_\alpha$,
  It sits at the heart of
  the coming application of the Geometric Vertex Decomposition Lemma.
}

\begin{Theorem}\label{thm:schubertCM}
  Each Schubert patch $X_w|_v$ is normal and Cohen-Macaulay.
\end{Theorem}

\begin{proof}
  We may assume $w\leq v$, for otherwise $X_w|_v$ is empty
  by proposition \ref{prop:cover}. 

  The proof is induction on $v$. 
  If $v=1$, then $w=1$, and $X_w|_v = N_- B/B$ is smooth.
  Otherwise we fix a simple root $\alpha$ with $v r_\alpha < v$
  (we know one exists by the last claim in proposition \ref{prop:easycases}, 
  though it is phrased there for $w$).

  By proposition \ref{prop:easycases},
  \begin{eqnarray*}
    w r_\alpha > w \qquad &\implies& \qquad
    X_w|_v \iso X_w|_{v r_\alpha}\qquad \text{albeit not $T$-equivariantly} \\
    w r_\alpha < w,\text{ but }v r_\alpha \not\!<w,\qquad &\implies& \qquad
    X_w|_v \times \AA^1_{-v\cdot\alpha} 
    \iso X_{w r_\alpha}|_v \iso X_{w r_\alpha}|_{v r_\alpha}
  \end{eqnarray*}
  and since $v r_\alpha < v$, by induction we know each right-hand side
  is normal and Cohen-Macaulay. (Indeed, we need that, to be able to invoke
  case (2) of proposition \ref{prop:easycases}.)

  We are now in the case $v > v r_\alpha \geq w > w r_\alpha$. 
  (In fact $v r_\alpha > w$ automatically, as otherwise we would have 
  $v > v r_\alpha > v$.) It is here that we will finally apply the
  Geometric Vertex Decomposition Lemma.

  Consider the open embedding
  $\gamma : Rad(P_{-\alpha}) \times \Pone \to G/B$
  from proposition \ref{prop:param}, and twist it by $v$.
  By proposition \ref{prop:cover}, the image of $v\cdot \gamma$ is
  $X_1|_{v,v r_\alpha}$.
  Let $\barX \subseteq Rad(P_{-\alpha}) \times \Pone$ be the
  preimage of $X_w$ under this map. In particular,
  $$ \barX \iso X_w \cap X_1|_{v,v r_\alpha} = X_w|_{v,v r_\alpha} $$
  so it is automatically reduced and irreducible.
  Our goal is to show that $\barX$ is normal and Cohen-Macaulay.
  
  {\bf Checking the conditions of the Geometric Vertex Decomposition Lemma.}
  To do so, we need first compute $\Pi$ and $\Lambda$. 
  Consider the two commuting squares
  $$ 
  \begin{array}{ccccccc}
    Rad(P_{-\alpha}) \times \Pone 
    & \stackrel{v \cdot \gamma}\longrightarrow & \ \ G/B 
    & \longleftarrow & X_w \\
    \downarrow & & \ \ \ \ \ \ \ \downarrow \pi_\alpha & & \downarrow \\
    Rad(P_{-\alpha}) &  \stackrel{v\cdot}\longrightarrow & \ \ G/P_\alpha
    & \longleftarrow & \pi_\alpha(X_w) & = & \pi_\alpha(X_{w r_\alpha})
  \end{array}
  $$
  where each $\rightarrow$ is an open embedding
  and each $\leftarrow$ is a closed embedding. Since each vertical map
  is (proper and) surjective, so too is the map from 
  the pullback of the top row, $\barX$, to the pullback of the bottom row.
  Hence the pullback of the bottom row is the $\Pi \subseteq Rad(P_{-\alpha})$
  we seek.

  As in proposition \ref{prop:cover}, the image of 
  $ Rad(P_{-\alpha}) \to G/P_\alpha$ is $v N_- P_\alpha / P_\alpha$,
  the ($v$ twist of the) big cell.  
  Hence the pullback $\Pi$ of the bottom row is 
  $v N_- P_{\alpha}/P_{\alpha}$ intersected with
  $\pi_\alpha(X_w) = \pi_\alpha(X_{w r_\alpha})$.
  By part (1) of proposition \ref{prop:easycases} applied to $w r_\alpha$
  (not $w$),
  $$ X_{w r_\alpha}|_{v r_\alpha}
  \iso \pi_\alpha(X_{w r_\alpha}|_{v,vr_\alpha}) \times \AA^1_{v\cdot\alpha}
  = \Pi \times \AA^1_{v\cdot\alpha}. $$
  By induction, the left-hand side is normal and Cohen-Macaulay, 
  so $\Pi$ is too.

  The last condition left to check on $\Pi$ is that
  $\barX \onto \Pi$ is generically $1$:$1$.
  By the assumption $w > w r_\alpha$, 
  the surjection $X_w \onto \pi_\alpha(X_w)$ is generically $1$:$1$, and 
  the map $\barX \onto \Pi$ is just a restriction of that to an open subset.

  Even before we compute $\Lambda$ exactly, we point out that 
  $(1,0) \stackrel{v\cdot\gamma}{\mapsto} vB \in X_w|_{v,vr_\alpha}$ 
  by the assumption $v\geq w$, 
  so $\barX \not\subseteq Rad(P_{-\alpha}) \times \{\infty\}$.

  Now we compute $\Lambda \times \{\infty\} 
  = \barX \cap (Rad(P_{-\alpha}) \times \{\infty\})$.
  Under the open embedding $v\cdot \gamma$,
  $$ \Lambda \times \{\infty\}) 
  \ \iso\ v\cdot \gamma(\Lambda \times \{\infty\}) 
  = X_w \ \cap\  v\cdot \gamma(Rad(P_{-\alpha}) \times \{\infty\}). $$
  Here 
  $$ \gamma(Rad(P_{-\alpha}) \times \{\infty\}) 
  = Rad(P_{-\alpha}) r_\alpha B/B
  = N_- r_\alpha B/B
  = X_{r_\alpha}^\circ $$
  so
  $$  \Lambda \times \{\infty\}) 
  \ \iso\ v\cdot \gamma(\Lambda \times \{\infty\}) 
  = X_w \ \cap\  v \cdot X_{r_\alpha}^\circ.
  $$
  Now we use lemma \ref{lem:conecomponent} to relate this space 
  $X_w \cap v \cdot X_{r_\alpha}^\circ$ to $X_w|_{v r_\alpha}$,
  which is reduced (and irreducible) since it is an open set in $X_w$,
  and is normal and Cohen-Macaulay by induction.

  With these conditions on $\Lambda$ and the ones already
  checked on $\Pi$, we may apply
  the Geometric Vertex Decomposition Lemma and lemma \ref{lem:gvdnormal}, 
  and see that $\barX$ is normal and Cohen-Macaulay.
\end{proof}

A stronger statement is known: 
for any nef line bundle $\mathcal L$ on $G/B$, the affine variety 
$\Spec \bigoplus_{n\in \naturals} H^0(X_w; \mathcal L^{\tensor n})$
is normal and Cohen-Macaulay. (While it would seem that this property
of $(X_w,\mathcal L)$ should be called ``affinely normal'' and
``affinely Cohen-Macaulay'', the regrettably standard terminology
is ``projectively normal'' and ``arithmetically Cohen-Macaulay''.)
We did not see how to derive these stronger statements with
the techniques of this paper.

We extract the following result from the above proof, filling in 
the case that was left open in proposition \ref{prop:easycases}.

\begin{Proposition}\label{prop:willems}
  Assume $v > v r_\alpha > w > w r_\alpha$. 
  Then there is a $T$-equivariant flat, locally free, degeneration 
  of $X_w|_v$ to the reduced scheme
  $$ 
  \left( \Pi \times \{0\} \right)
  \ \cup_{\Lambda \times \{\vec 0\}} \ 
  \left( \Lambda \times \AA^1_{v\cdot \alpha} \right)
  $$
  where 
  $$ 
  \Pi \times \AA^1_{v\cdot \alpha} \iso X_{w r_\alpha}|_{vr_\alpha} 
  \qquad \text{and} \qquad
  \Lambda \times \AA^1_{-v\cdot \alpha} \iso X_w|_{v r_\alpha},
  \qquad \text{$T$-equivariantly}. $$
\end{Proposition}

To see that the family constructed in theorem \ref{thm:GVDs} is not just
flat but locally free, we note that the $T$-action on the fibers
contains a one-parameter subgroup $v\cdot \check \rho$ acting with
only positive weights. Taking various GIT quotients, one may reduce to
the projective case, where flat families are automatically locally free.

\subsection{Connections to subword complexes and the Billey-Willems formula}
\label{ssec:subword}
The proof of theorem \ref{thm:schubertCM} used a simple root $\alpha$
with $v r_\alpha < v$. 
To unroll the induction, then, one needs a {\dfn reduced word} for $v$, 
which is a minimal sequence $Q = \{\alpha_1,\alpha_2,\ldots,\alpha_k\}$ 
of simple roots such that $v = \prod_{i=1}^k r_{\alpha_i}$.
(Careful: the first root used in the proof is then $\alpha_k$, not $\alpha_1$.)

In \cite{subword}, we associated a simplicial complex 
$\Delta(Q,w)$ to a (not necessarily reduced) word $Q$ and a Weyl
group element $w$, called the {\dfn subword complex}.
The vertices are the elements of $Q$, and a subset of $Q$ is a facet
(maximal face) if its complement is a reduced word for $w$. 
In particular all the facets have the same dimension.

\begin{Theorem}\label{thm:subword}
  Let $w\leq v$ in the Bruhat order, 
  and let $Q$ be a reduced expression for $v$. 
  \begin{enumerate}
  \item \cite{Billey} The restriction of the equivariant
    cohomology class $[X_w] \in H^*_T(G/B)$ to the point $v$
    can be computed as a sum over the facets of $\Delta(Q,w)$.
  \item \cite{Willems} The restriction of the equivariant
    $K$-class $[X_w] \in K^*_T(G/B)$ to the point $v$
    can be computed as an alternating sum over 
    the interior faces of $\Delta(Q,w)$.
  \item $X_w|_v$ has a $T$-equivariant flat locally free degeneration to 
    the Stanley-Reisner scheme of (an irrelevant multicone on) $\Delta(Q,w)$.
  \end{enumerate}
\end{Theorem}

All the geometric groundwork has been laid in propositions
\ref{prop:easycases} and \ref{prop:willems}. The details of the bookkeeping 
will appear elsewhere \cite{Kostant}, but we include a sketch here.

\begin{proof}[Proof sketch]
  The complex $\Delta(Q,w)$ has a ``vertex decomposition'' into two
  subcomplexes, depending on whether one's subword uses the last letter
  in $Q$ or not. Things are simple when the last letter is required
  or forbidden, and more interesting when it is optional.

  The simple cases exactly match proposition \ref{prop:easycases},
  and the interesting case matches the geometric vertex decomposition
  in proposition \ref{prop:willems}. This, and induction, prove
  the third claim. Since the $K$-class and cohomology class are
  invariant in locally free $T$-equivariant families, they can be
  computed from the subword complex.
  
  In \cite{subword} we prove that the subword complex is homeomorphic
  to a ball, so its $K$-class can be computed as an alternating sum over 
  the interior faces. We characterize those faces in a way that exactly
  matches the terms in Willems' formula. Then either the degeneration,
  or Willems' formula, imply Billey's formula.
\end{proof}

This degeneration to the Stanley-Reisner scheme of a shellable
simplicial ball generalizes ones from \cite{KodRag,GhoRag,KreLak}
which applied to Schubert patches in various types of Grassmannians.
Very loosely speaking, in either situation one needs a certain
multiplicity-freeness to ensure even that the degeneration is
generically reduced. In the Grassmannian cases the Pl\"ucker embedding
is into a minuscule representation, and the minusculeness provides the
multiplicity-freeness.  Here the multiplicity-freeness comes from
lemma \ref{lem:conecomponent} and from the fact that the projection
$X_w \to \pi_\alpha(X_w)$ is generically $1$:$1$ for $w>w r_\alpha$.

\bibliographystyle{alpha}    

\end{document}